\documentclass[10pt,leqno]{amsart}
\usepackage{amsmath,amssymb,amsthm,mathtools,mathrsfs}
\usepackage{microtype}
\usepackage[colorlinks=true,linkcolor=blue,citecolor=blue,urlcolor=blue]{hyperref}
\numberwithin{equation}{section}
\newtheorem{theorem}{Theorem}[section]
\newtheorem{lemma}[theorem]{Lemma}
\newtheorem{proposition}[theorem]{Proposition}
\newtheorem{corollary}[theorem]{Corollary}
\theoremstyle{definition}

\theoremstyle{remark}
\newtheorem{remark}[theorem]{Remark}
\newcommand{\R}{\mathbb R}
\newcommand{\T}{\mathbb T}
\newcommand{\Ric}{\operatorname{Ric}}
\newcommand{\tr}{\operatorname{tr}}
\newcommand{\Sym}{\operatorname{Sym}}

\allowdisplaybreaks[1]
\title[Li-Yau and Harnack estimates]{Li-Yau and Harnack estimates for nonlinear parabolic equations}
\author[Jeffres]{Thalia Jeffres}
\address{Department of Mathematics and Statistics, Wichita State University,
Wichita, KS 67260-0033, USA}
\email{thalia.jeffres@wichita.edu}
\author[Li]{Xiaolong Li}
\address{Department of Mathematics and Statistics, Auburn University,
Auburn, AL 36849, USA}
\email{xil0005@auburn.edu}
\thanks{XL's research is partially supported by NSF-DMS \#2553660 and a
start-up grant at Auburn University.}
\date{}
\begin{document}
\begin{abstract}
We establish global Harnack inequalities for positive viscosity solutions
of several classes of nonlinear parabolic equations. The equations include positively homogeneous uniformly parabolic flows on Euclidean space and flat tori, normalized weighted
$p$-Laplacian heat flows, and homogeneous isotropic diffusions on closed smooth metric measure spaces. For the normalized infinity-Laplacian heat flow, no curvature assumption is needed.
For the ordinary weighted $p$-Laplacian, we obtain pressure
estimates under nonnegative Bakry-\'{E}mery Ricci lower bounds.
The proofs combine comparisons at two independent space-time points using the maximum principle for semicontinuous functions with weighted geodesic variations. The corresponding Li-Yau inequalities then follow in the first-order viscosity sense.
\end{abstract}
\maketitle

% BEGIN SECTION 1
\section{Introduction}\label{sec:introduction}

Harnack inequalities relate the values of a positive solution of a partial differential equation at different points and provide a basic tool for studying regularity, heat kernels,
and the geometry of the underlying space. For parabolic equations, the
comparison also reflects the direction of time. Moser~\cite{Moser1964}
proved a parabolic Harnack inequality for uniformly parabolic equations
in divergence form with measurable coefficients. His integral method
allows one to obtain regularity without differentiating the coefficients.
For the heat equation on Riemannian manifolds, Li and
Yau~\cite{LiYau1986} revealed a more precise relation between Harnack
inequalities and curvature. Their estimate has become a basic tool in
geometric analysis.

Let $(M^n,g)$ be a complete Riemannian manifold with nonnegative Ricci
curvature, and let $u:M \times (0,\infty) \to \R$ be a positive solution of the heat equation $u_t=\Delta u$.
Li and Yau~\cite{LiYau1986} proved the differential Harnack estimate
\begin{equation}\label{intro:classical-li-yau}
 \frac{|\nabla u|^2}{u^2}-\frac{u_t}{u}\leq\frac{n}{2t}
\end{equation}
for all $(x,t) \in M \times (0, \infty)$.
Integrating this inequality along space-time curves gives
\begin{equation}\label{intro:classical-harnack}
 u(x,t)\geq u(y,s)\left(\frac{s}{t}\right)^{\frac{n}{2}}
 \exp\!\left(-\frac{d(x,y)^2}{4(t-s)}\right),
\end{equation}
for all $x,y\in M$ and $0<s<t$.
Both the dimension in the time factor and the coefficient of the
distance term are determined by the Euclidean heat kernel. In addition
to its consequences for positive solutions, this estimate leads to heat
kernel bounds and eigenvalue estimates; see \cite{LiYau1986,Davies1989}.
It also provides a model for differential Harnack inequalities in
geometric evolution equations; see Hamilton~\cite{HamiltonRicci1993}
and Chow~\cite{Chow1992}.

There have been many developments of the Li-Yau estimate under other
curvature assumptions. Davies~\cite{Davies1989} refined the estimate
under a negative Ricci lower bound. Hamilton~\cite{Hamilton1993}
established a matrix Harnack inequality under stronger curvature
hypotheses, while Souplet and Zhang~\cite{SoupletZhang2006} obtained
localized spatial gradient estimates and applications to Liouville
theorems. Li and Xu~\cite{LiXu2011} derived differential Harnack
inequalities with time-dependent coefficients under a Ricci lower
bound. Bakry, Bolley, and Gentil~\cite{BakryBolleyGentil2017}
proved a refined Li-Yau inequality in the more general setting of
diffusion semigroups satisfying a curvature-dimension condition.
These results show that the curvature contribution can be expressed in
different ways, and that the choice of the coefficients is important
both for small times and for large times.
More recently, Q.~S.~Zhang~\cite{ZhangSharp2025} proved a Li-Yau
gradient bound on compact manifolds without assuming nonnegative Ricci
curvature. The coefficient of $\partial_t\log u$ remains one, while
the upper bound also involves the curvature and diameter. Thus the
gradient-time combination in \eqref{intro:classical-li-yau} can still
be controlled when the Ricci curvature has a negative part, although
the right-hand side changes. Zhang's survey~\cite{ZhangSurvey}
discusses these developments and related gradient and heat kernel
estimates. The distinction between retaining the coefficient one and
retaining the precise bound \(\frac{n}{2t}\) is relevant when comparing
extensions of the classical inequality.

For nonlinear diffusion, the appropriate differential quantity need
not be the logarithm of the solution. The pressure estimates of
Aronson and B\'enilan~\cite{AronsonBenilan1979} for the porous medium
equation initiated a corresponding theory in which the time scale
depends on the homogeneity of the equation. Esteban and
V\'azquez~\cite{EstebanVazquez1990} studied pressure estimates for
the parabolic $p$-Laplacian. Auchmuty and Bao~\cite{AuchmutyBao1994}
derived Harnack inequalities by integrating more general differential
inequalities and optimizing the resulting space-time cost. On
Riemannian manifolds, Kotschwar and Ni~\cite{KotschwarNi2009}
established gradient estimates for $p$-harmonic functions and related
parabolic equations. Lu, Ni, V\'azquez, and
Villani~\cite{LuNiVazquezVillani2009} developed local
Aronson-B\'enilan estimates for porous medium and fast diffusion
equations on manifolds. Weighted differential Harnack and entropy
estimates were obtained by Wang, Yang, and
Chen~\cite{WangYangChen2013} and Wang~\cite{Wang2017,Wang2021}.
We will make a precise comparison with Wang's weighted pressure
estimate below.

The nonlinear setting presents an additional difficulty. Even for
smooth initial data, solutions of degenerate or singular equations
need not have the derivatives used in a classical Bochner argument.
For nondivergence-form equations, viscosity solutions provide a
natural framework, but a differential estimate must then be interpreted
through test functions. At a critical point of the solution, the
operator itself may be singular or depend on the direction of the
gradient. Our interest is in explicit global comparisons of the form
\eqref{intro:classical-harnack}, or their pressure counterparts, for
which the spatial cost and time weight are specified.

A different way to obtain such estimates is to compare the solution
directly at two space-time points. Andrews, Hauer, and
Slegers~\cite{AndrewsHauerSlegers2026} developed this approach for
linear Schr\"odinger equations. They first prove the global Harnack
inequality by a multi-point maximum principle and then recover the
differential inequality by letting the two points approach one another.
Their work suggests that the same order of argument may be useful
when differentiating the evolution equation is not available.
The spatial part of the comparison is closely related to the
modulus-of-continuity method of Andrews and
Clutterbuck~\cite{AndrewsClutterbuck2009,AndrewsClutterbuck2013}.
Its viscosity and weighted-manifold versions were developed in
\cite{LiWang2017,Li2021,LiTuWang2024, JeffresLi2025}, among others.

The purpose of this paper is to establish such direct comparisons for
several classes of nonlinear viscosity solutions. We treat positively
homogeneous uniformly elliptic operators, normalized weighted
$p$-Laplacians, and homogeneous isotropic operators that include the
ordinary weighted $p$-Laplacian. The logarithm is appropriate for
degree-one diffusion, while a power of the solution gives the pressure
for the other homogeneities. The main point of the proofs is to choose
the space-time cost and the weights in the endpoint variations so
that their second-order contributions agree. The maximum principle for semicontinuous
functions~\cite{CIL} then replaces differentiation of the solution. This gives
the first-order differential inequality in the viscosity sense as a
consequence of the global comparison. Several of the estimates are
already known in particular settings, including for viscosity
solutions; the statements below
specify the solution classes, operators, and geometric assumptions
covered by the present argument.

We begin with homogeneous uniformly parabolic equations of the form
\begin{equation}\label{eq 1.1 fully nonlinear}
 u_t=F(D^2u).
\end{equation}
Here $F\colon\Sym(n)\to\R$ is continuous, positively homogeneous
of degree one, and uniformly elliptic with constants
$0<\lambda\leq\Lambda$. Thus
\begin{equation}\label{intro:fn-homogeneity}
 F(qX)=qF(X),
\end{equation}
for $q\geq0$ and $X\in\Sym(n)$, and
\begin{equation}\label{intro:fn-ellipticity}
 \lambda\tr Y\leq F(X+Y)-F(X)\leq\Lambda\tr Y,
\end{equation}
for $X,Y\in\Sym(n)$ with $Y\geq0$.
No convexity or concavity of $F$ is assumed. On a flat torus we use
a fixed global parallel orthonormal frame, and $d$ denotes its flat
distance. All differential inequalities for logarithms and pressures
below are understood in the first-order viscosity sense: they hold
for every smooth test function touching the relevant function from
below. The second-order solution conventions, including those at
critical gradients, are given in Section~\ref{sec:prelim}.

\begin{theorem}\label{thm: main fully nonlinear}
Let $M=\R^n$ or a flat torus $\T^n$, and let
$F\colon\Sym(n)\to\R$ be continuous and satisfy
\eqref{intro:fn-homogeneity} and \eqref{intro:fn-ellipticity}
with constants $0<\lambda\leq\Lambda$.
Suppose that $u$ is a positive continuous viscosity
solution of \eqref{eq 1.1 fully nonlinear} on $M\times(0,T)$.
If $M=\R^n$, assume that $u$ is bounded and uniformly continuous
on every strip $\R^n\times[a,b]$, for $0<a<b<T$. Then
\begin{equation}\label{eq 1.2 harnack}
 u(x,t)\geq u(y,s)\left(\frac{s}{t}\right)^{\frac{n\Lambda}{2\lambda}}
 \exp\!\left(-\frac{d(x,y)^2}{4\lambda(t-s)}\right)
\end{equation}
for all $x,y\in M$ and $0<s<t<T$. Moreover, $v=\log u$ satisfies
\begin{equation}\label{eq 1.3 differential Harnack}
 v_t-\lambda|Dv|^2+\frac{n\Lambda}{2\lambda t}\geq0
\end{equation}
on $M\times (0,T)$.
\end{theorem}

When $F(X)=\tr X$ and $\lambda=\Lambda=1$, this recovers
the Li-Yau Harnack inequality \eqref{intro:classical-harnack}
in the stated Euclidean solution class and on flat tori. For a general
$F$, \eqref{eq 1.2 harnack} gives an
explicit estimate in terms of the ellipticity constants, but we do not
claim that its time exponent is optimal for each operator. The
operator-dependent self-similar exponents studied by Armstrong and
Trokhimtchouk~\cite{ArmstrongTrokhimtchouk2010} provide a related
perspective on the role of homogeneity. The proof of
Theorem~\ref{thm: main fully nonlinear} uses a Pucci comparison
between the two scaled endpoint Hessians, so it does not require
a linearization of $F$.

We also extend these estimates to closed, connected Riemannian
manifolds for orthogonally invariant operators. The same bounds hold
under nonnegative sectional curvature, while a negative sectional
curvature lower bound leads to a modified time factor and a
corresponding differential inequality. The full statements and proofs
are given in Section~\ref{sec:harnack-manifolds}; see
Theorems~\ref{intro:thm-manifold-nonnegative}
and~\ref{intro:thm-manifold-negative}.

We next consider a closed, connected smooth metric measure space
$(M^n,g,e^{-f}d\mu_g)$, where $f$ is smooth. For $n<N<\infty$, write
\[
 \Ric_f^N=\Ric+\nabla^2f-\frac{df\otimes df}{N-n}.
\]
When $N=n$, we require $f$ to be constant and set $\Ric_f^n=\Ric$.
This convention will be used in every weighted result below. The
weighted Laplacian and normalized infinity Laplacian are given by
\[
 \Delta_f u=\Delta u-\langle\nabla f,\nabla u\rangle,
\]
and
\[
 \Delta_\infty^{\mathrm N}u
 =\frac{\nabla^2u(\nabla u,\nabla u)}{|\nabla u|^2},
\]
respectively, where the second formula applies at noncritical points. We use
\[
 \Delta_{p,f}^{\mathrm N}u
 =\Delta_f u+(p-2)\Delta_\infty^{\mathrm N}u,
\]
for $1<p<\infty$.
The superscript $\mathrm N$ denotes normalization, independently
of the dimension parameter $N$.\footnote{Our normalization does not include
a factor $\frac{1}{p}$; this fixes the time scale in the estimates.}

For the unweighted normalized flow, Banerjee, Garofalo, and
Mahmoudian~\cite{BanerjeeGarofaloMahmoudian2026} recently proved a
sharp Li-Yau estimate and its global Harnack consequence for viscosity
solutions on closed manifolds with nonnegative Ricci curvature.
Their theorem assumes smooth positive initial data, and their proof proceeds
through uniformly parabolic approximations and differential
estimates. Our next result recovers their constants by a direct
two-time argument when $f$ is constant and $N=n$, permits continuous
positive initial data, and uses finite-dimensional weighted Ricci curvature.

\begin{theorem}\label{thm:intro-normalized-p}
Assume that $\Ric_f^N\geq0$ for $n\leq N<\infty$. Let
$1<p<\infty$, and let $u\in C(M\times[0,T))$ be a positive
viscosity solution of $u_t=\Delta_{p,f}^{\mathrm N}u$. Then
\begin{equation}\label{intro:normalized-p-harnack}
 u(x,t)\geq u(y,s)
 \left(\frac{s}{t}\right)^{\frac{N+p-2}{2(p-1)}}
 \exp\!\left(-\frac{d(x,y)^2}{4(p-1)(t-s)}\right)
\end{equation}
for all $x,y\in M$ and $0<s<t<T$. Moreover, $v=\log u$ satisfies
\begin{equation}\label{intro:normalized-p-differential}
 v_t-(p-1)|\nabla v|^2+\frac{N+p-2}{2(p-1)t}\geq0
\end{equation}
on $M\times(0,T)$.
\end{theorem}

The longitudinal and transverse directions contribute separately to
the exponent in \eqref{intro:normalized-p-harnack}. This is useful
for the normalized infinity heat equation, where the operator has
only a longitudinal second derivative. Juutinen and
Kawohl~\cite{JuutinenKawohl2006} established existence, regularity,
and a local Harnack inequality for this equation in Euclidean space.
The following result gives an explicit global comparison on closed
manifolds without a curvature assumption.

\begin{theorem}\label{thm:intro-infinity}
Let $(M^n,g)$ be any closed, connected Riemannian manifold. If
$u\in C(M\times[0,T))$ is a positive viscosity solution of
$u_t=\Delta_\infty^{\mathrm N}u$, then
\begin{equation}\label{intro:infinity-harnack}
 u(x,t)\geq u(y,s)\left(\frac{s}{t}\right)^{\frac{1}{2}}
 \exp\!\left(-\frac{d(x,y)^2}{4(t-s)}\right)
\end{equation}
for all $x,y\in M$ and $0<s<t<T$, and $v=\log u$ satisfies
\begin{equation}\label{intro:infinity-differential}
 v_t-|\nabla v|^2+\frac{1}{2t}\geq0.
\end{equation}
\end{theorem}

The absence of curvature in Theorem~\ref{thm:intro-infinity} follows
from the use of longitudinal variations alone. More generally, for constants $A>0$,
$B\geq0$, and $\gamma>0$, define the operator
\begin{equation}\label{intro:redraft-operator}
 \mathcal Q_\gamma[u]
 =|\nabla u|^{\gamma-1}
   \left((A-B)\Delta_\infty^{\mathrm N}u+B\Delta_f u\right)
\end{equation}
at noncritical points. This includes normalized $p$ diffusion when
$(\gamma,A,B)=(1,p-1,1)$ and ordinary weighted $p$ diffusion when
$(\gamma,A,B)=(p-1,p-1,1)$. The next proposition contains the
normalized $p$ and infinity heat estimates as cases of degree-one
diffusion with arbitrary coefficients.

\begin{proposition}\label{intro:prop-degree-one}
Let $(M^n,g,e^{-f}d\mu_g)$ be closed and connected, and let
$u\in C(M\times[0,T))$ be a positive viscosity solution of
$u_t=\mathcal Q_1[u]$. If $B>0$, assume $\Ric_f^N\geq0$
for $n\leq N<\infty$ and set
\[
 \eta_N=\frac{1}{2}+\frac{(N-1)B}{2A}.
\]
Then
\[
 u(x,t)\geq u(y,s)\left(\frac{s}{t}\right)^{\eta_N}
 \exp\!\left(-\frac{d(x,y)^2}{4A(t-s)}\right),
\]
for all $x,y\in M$ and $0<s<t<T$. Moreover, $v=\log u$ satisfies
\[
 v_t-A|\nabla v|^2+\frac{\eta_N}{t}\geq0
\]
on $M\times(0,T)$.
If $B=0$, these conclusions
hold with $\eta_N=\frac{1}{2}$ on any closed, connected Riemannian
manifold, without a curvature assumption.
\end{proposition}

We prove Proposition~\ref{intro:prop-degree-one} for arbitrary
$A>0$ and $B\geq0$; the normalized $p$ and infinity estimates
then follow by substitution.
For $\gamma\ne1$, the pressure variable changes both the time
weight and the power of the distance. The isotropic class
\eqref{intro:redraft-operator} is part of the setting studied in
\cite{ParviainenVazquez2020,LiTuWang2024,KurkinenSiltakoski2025}.
The following theorem gives a global pressure comparison for this
class on weighted manifolds.

\begin{theorem}\label{intro:homogeneous-pressure}
Let $(M^n,g,e^{-f}d\mu_g)$ be a closed, connected smooth metric
measure space satisfying $\Ric_f^N\geq0$, where $n\leq N<\infty$
and $f$ is constant when $N=n$. Let
$u\in C(M\times[0,T))$ be a positive viscosity solution of
$u_t=\mathcal Q_\gamma[u]$, where $A>0$, $B\geq0$,
$\gamma>0$, and $\gamma\ne1$. Assume that
$2A+(\gamma-1)(N-1)B>0$, and set
\[
 \theta=\frac{\gamma-1}{\gamma},\quad
 v=\frac{u^\theta}{\theta},\quad
 \kappa=\frac{A+\gamma(N-1)B}
 {\gamma\bigl(2A+(\gamma-1)(N-1)B\bigr)}.
\]
Define
\[
 \Lambda(s,t)=\left(
 \frac{A(\gamma+1)}{\gamma}
 \int_s^t\tau^{-(\gamma-1)\kappa}\,d\tau
 \right)^{-\frac{1}{\gamma}}.
\]
Then
\begin{equation}\label{intro:general-pressure-harnack}
 t^{\theta\kappa}v(x,t)-s^{\theta\kappa}v(y,s)
 +\frac{\gamma}{\gamma+1}\Lambda(s,t)
       d(x,y)^{\frac{\gamma+1}{\gamma}}\geq0
\end{equation}
for all $x,y\in M$ and $0<s<t<T$, and
\[
 v_t-\frac{A}{\gamma}|\nabla v|^{\gamma+1}
       +\frac{\theta\kappa}{t}v\geq0
\]
on $M\times (0,T)$. At a classical noncritical point, the latter inequality is equivalent
to $\mathcal Q_\gamma[v]\geq-\frac{\kappa}{t}$.
If $B=0$, the curvature assumption can be omitted.
\end{theorem}

The range condition is automatic for $\gamma>1$; for
$0<\gamma<1$, it is equivalent to
$(1-\gamma)(N-1)B<2A$. For $B>0$ and $N=n$ in the
unweighted Euclidean setting,
this is the same singular range that appears in the intrinsic Harnack
theory of~\cite{KurkinenSiltakoski2025}, after matching the parameters
of the two operators. The role of the quantity
$1+\frac{(n-1)\gamma B}{A}$ is also consistent with the effective
dimension in the radial reductions of~\cite{ParviainenVazquez2020}.
For this operator class, Theorem~\ref{intro:homogeneous-pressure}
gives the explicit global comparison
\eqref{intro:general-pressure-harnack} for viscosity solutions on
weighted manifolds.
For $B=0$, we have $\kappa=\frac{1}{2\gamma}$, and the result
holds on an arbitrary closed manifold.
Homogeneous infinity diffusion in the range $\gamma>1$ was studied
by Portilheiro and V\'azquez~\cite{PortilheiroVazquez2013}, including
explicit solutions and large-time behavior.

For the ordinary weighted $p$-Laplacian
\[
 \Delta_{p,f}u
 =e^f\operatorname{div}\left(e^{-f}|\nabla u|^{p-2}\nabla u\right),
\]
Theorem~\ref{intro:homogeneous-pressure} takes the following form.
We state this specialization separately to make the comparison with
the pressure estimates in the literature explicit.

\begin{theorem}\label{thm:intro-weighted-p}
Let $(M^n,g,e^{-f}d\mu_g)$ be closed and connected, with
$\Ric_f^N\geq0$ for $n\leq N<\infty$, and with $f$ constant
when $N=n$. Let $u\in C(M\times[0,T))$ be a positive viscosity
solution of $u_t=\Delta_{p,f}u$, for
$p>\frac{2N}{N+1}$ and $p\ne2$. Define
\[
 \theta=\frac{p-2}{p-1},\quad v=\frac{u^\theta}{\theta},\quad
 \kappa=\frac{N}{p+N(p-2)},\quad
 r=\frac{p}{p+N(p-2)}.
\]
Then
\begin{equation}\label{intro:weighted-p-harnack}
 t^{\theta\kappa}v(x,t)-s^{\theta\kappa}v(y,s)
 +\frac{p-1}{p}
 \left(\frac{r}{p(t^r-s^r)}\right)^{\frac{1}{p-1}}
 d(x,y)^{\frac{p}{p-1}}\geq0
\end{equation}
for all $x,y\in M$ and $0<s<t<T$. Moreover,
\begin{equation}\label{intro:weighted-p-differential}
 v_t-|\nabla v|^p+\frac{\theta\kappa}{t}v\geq0.
\end{equation}
At a classical noncritical point,
\eqref{intro:weighted-p-differential} is equivalent to
$\Delta_{p,f}v\geq-\frac{\kappa}{t}$.
\end{theorem}

The constant $\kappa$ in this theorem has an established history.
The Euclidean pressure estimate is associated with the work of
Esteban and V\'azquez~\cite{EstebanVazquez1990}, and the geometric
pressure method is developed in~\cite{KotschwarNi2009}.
More directly, Wang~\cite[Theorem~1.11]{Wang2021}, specialized
to the ordinary weighted $p$-heat equation, gives the same constant
$\frac{N}{p+N(p-2)}$ for smooth pressure solutions on closed
weighted manifolds with nonnegative
finite-dimensional Bakry-\'{E}mery curvature, in both ranges
$p>2$ and $\frac{2N}{N+1}<p<2$.
The slow-diffusion case also follows from the more general
reaction-diffusion estimate of Wang and
Wang~\cite{WangWang2020} by taking the power inside the diffusion
to be one and the reaction coefficient to be zero.
The contribution of Theorem~\ref{thm:intro-weighted-p} is a direct
proof of the global comparison for the
stated class of continuous viscosity solutions, with a first-order
viscosity conclusion that includes critical points. The conjugate
power of the distance and the optimized time cost also have
antecedents in~\cite{AuchmutyBao1994}; here they are built into
the comparison argument itself.

On closed smooth metric measure spaces with $\Ric_f^N\geq-Kg$,
we also obtain a global pressure Harnack inequality and an associated
differential estimate for the weighted $p$-heat equation, for
$p>\frac{2N}{N+1}$ and $p\ne2$. The time weight is hyperbolic,
and the constants depend on the initial amplitude without an
a priori gradient bound. The full statement and proof are given
in Theorem~\ref{intro:negative-pressure} in
Section~\ref{sec:negative-curvature}, where the same method also
gives a weighted heat estimate for $p=2$.

There are two main technical points in the proofs. First, the
comparison uses independent endpoint times. We apply the maximum principle
for semicontinuous functions to the full space-time variables and retain the appropriate
parabolic inequalities at each endpoint. The weights in the endpoint
variations are chosen according to the homogeneity of the operator.
Second, on a manifold the distance need not be smooth at a
contact pair. A smooth upper support constructed along a
minimizing geodesic allows the second variation argument to be
used there. The transverse variations produce the curvature term,
while the longitudinal variation determines the spatial cost.
The purely longitudinal equations therefore require no curvature
assumption. For singular equations, the treatment of zero gradients
and the required positive-time regularity are part of the viscosity
argument.

The paper is organized as follows. Section~\ref{sec:prelim} fixes the viscosity
conventions and proves the auxiliary jet and regularity results.
Section~\ref{sec:fully} treats homogeneous uniformly elliptic equations on
Euclidean space and flat tori. Section~\ref{sec:homogeneous} proves the degree-one comparison and
deduces the normalized $p$ and infinity heat estimates.
Section~\ref{sec:homogeneous-power} considers general
homogeneous isotropic diffusion and its weighted $p$-Laplacian
specialization. Section~\ref{sec:negative-curvature} treats a negative weighted Ricci lower
bound.
Section~\ref{sec:harnack-manifolds} establishes the fully nonlinear estimates on curved
manifolds.

\vskip0.5cm

\textbf{AI disclosure.} OpenAI’s ChatGPT assisted the authors in developing this paper, particularly Theorem \ref{thm:intro-weighted-p} and Theorem \ref{intro:negative-pressure}. All proofs are complete and all arguments and calculations are easy to verify by hand. The authors take full responsibility for the mathematical content.

% END SECTION 1

% BEGIN MAIN PROOFS
\section{Preliminaries}\label{sec:prelim}

We fix the viscosity conventions and prove two facts used in the
comparison arguments. The first permits the two endpoint times to
vary independently. The second gives positive-time Lipschitz
regularity in the singular range, including the purely longitudinal
case.

We write $\Sym(n)$ for the space of real symmetric $n\times n$
matrices. The uniform ellipticity assumption on $F$ is equivalent to
\[
 \mathcal M^-_{\lambda,\Lambda}(X-Y)
 \leq F(X)-F(Y)
 \leq\mathcal M^+_{\lambda,\Lambda}(X-Y),
\]
for $X,Y\in\Sym(n)$, where
\[
 \mathcal M^-_{\lambda,\Lambda}(Z)
 =\min_{\lambda I\leq H\leq\Lambda I}\tr(HZ),
\]
and
\[
 \mathcal M^+_{\lambda,\Lambda}(Z)
 =\max_{\lambda I\leq H\leq\Lambda I}\tr(HZ).
\]

On a Riemannian manifold, $\nabla$, $\nabla^2$, and
$\Delta=\tr_g\nabla^2$ denote the gradient, Hessian, and
Laplace-Beltrami operator. Norms and inner products are taken with
respect to $g$, and $d$ denotes the Riemannian distance. A closed
manifold is compact and has no boundary. We retain the convention
from the Introduction that $f$ is constant when the
curvature-dimension parameter $N$ equals $n$. We also write
$\Ric_f=\Ric+\nabla^2f$.

For the operator in \eqref{intro:redraft-operator}, it is convenient
to set
\[
 \mathcal Q_\gamma(x,q,X)
 =|q|^{\gamma-1}
 \left(B\tr_gX+(A-B)\frac{X(q,q)}{|q|^2}
                  -B\langle\nabla f,q\rangle\right),
\]
for $q\ne0$.
Here $q\in T_xM$ and $X$ is a symmetric bilinear form on $T_xM$;
we identify vectors and covectors using the metric. The parameters
satisfy $A>0$, $B\geq0$, and $\gamma>0$. At noncritical points,
\begin{align*}
 \Delta_{p,f}^{\mathrm N}u
 &=\mathcal Q_1[u],\quad\text{for }A=p-1,\ B=1,\\
 \Delta_\infty^{\mathrm N}u
 &=\mathcal Q_1[u],\quad\text{for }A=1,\ B=0,\\
 \Delta_{p,f}u
 &=\mathcal Q_{p-1}[u],\quad\text{for }A=p-1,\ B=1.
\end{align*}
Thus $\Delta_{p,f}u=|\nabla u|^{p-2}\Delta_{p,f}^{\mathrm N}u$
there.
With this convention, $\frac{1}{p-1}\Delta_{p,f}^{\mathrm N}$
converges to $\Delta_\infty^{\mathrm N}$ at nonzero gradients
as $p\to\infty$; the drift disappears under this time rescaling.

\subsection{Viscosity solutions}\label{subsec:viscosity}

We use the test-function and semijet conventions of~\cite{CIL}.
Unless stated otherwise, a test function is $C^2$ in space and time,
and its contact with the solution occurs at an interior space-time
point. Spatial derivatives on a manifold are covariant derivatives.
An upper test touches from above, and a lower test touches from
below, with equality at the contact point.

For $u_t=F(D^2u)$, an upper semicontinuous function $u$ is a
viscosity subsolution if every upper test $\varphi$ satisfies
\[
 \varphi_t\leq F(D^2\varphi)
\]
at the contact point. A lower semicontinuous function is a
viscosity supersolution if the reverse inequality holds at every
lower test. A continuous function is a viscosity solution when it
is both a subsolution and a supersolution. Prescribed initial data
are attained continuously.

For $u_t=\mathcal Q_\gamma[u]$, the inequalities at nonzero-gradient
tests are
\[
 \varphi_t\leq
 \mathcal Q_\gamma(x,\nabla\varphi,\nabla^2\varphi)\quad\text{at an upper test},
\]
and
\[
 \varphi_t\geq
 \mathcal Q_\gamma(x,\nabla\varphi,\nabla^2\varphi)\quad\text{at a lower test}.
\]
At critical points, we distinguish the following three cases.

If $\gamma>1$, the operator has the continuous extension
$\mathcal Q_\gamma(x,0,X)=0$. Hence a zero-gradient upper test
satisfies $\varphi_t\leq0$, and a zero-gradient lower test satisfies
$\varphi_t\geq0$.
If $\gamma=1$, the upper and lower semicontinuous envelopes are
\[
 \mathcal Q_1^*(x,0,X)
 =B\tr_gX+\max_{|e|=1}(A-B)X(e,e),
\]
and
\[
 \mathcal Q_{1,*}(x,0,X)
 =B\tr_gX+\min_{|e|=1}(A-B)X(e,e).
\]
At zero-gradient contacts we require
\[
 \varphi_t\leq\mathcal Q_1^*(x,0,\nabla^2\varphi)\quad\text{at an upper test},
\]
and
\[
 \varphi_t\geq\mathcal Q_{1,*}(x,0,\nabla^2\varphi)\quad\text{at a lower test}.
\]
For $A=1$ and $B=0$, the two envelopes are respectively the
largest and smallest eigenvalues of $X$.
If $0<\gamma<1$, we use admissible tests, as
in~\cite[Section~2]{KurkinenSiltakoski2025}. A zero-gradient test
at $(x_0,t_0)$ is admissible if, in a neighborhood of that point,
\[
 |\varphi(x,t)-\varphi(x_0,t_0)-\varphi_t(x_0,t_0)(t-t_0)|
 \leq\omega(d(x,x_0))+o(|t-t_0|),
\]
where the time remainder is independent of $x$ and the nonnegative
function $\omega\in C^2([0,r_0))$ satisfies
$\omega(0)=\omega'(0)=\omega''(0)=0$,
$\omega''(\ell)>0$ for $0<\ell<r_0$, and
\[
 (\omega'(\ell))^{\gamma-1}
 \left(\omega''(\ell)+\frac{\omega'(\ell)}{\ell}\right)
 \longrightarrow0
\]
as $\ell\downarrow0$.
We impose $\varphi_t\leq0$ at an admissible upper test and
$\varphi_t\geq0$ at an admissible lower test. No condition is
imposed at other zero-gradient tests. In particular,
$\omega(\ell)=C\ell^q$ is admissible for $C>0$ and
$q>1+\frac{1}{\gamma}$. Indeed, its second-order contribution is
of order $\ell^{\gamma(q-1)-1}$, while the drift is bounded by
$B|\nabla f|(\omega')^\gamma$; both tend to zero.

We recall the jet notation used below. In normal coordinates at
$x$, a triple $(a,q,X)$ belongs to the parabolic superjet
$\mathcal P^{2,+}u(x,t)$ if
\[
 u(\exp_x\xi,t+\tau)
 \leq u(x,t)+a\tau+\langle q,\xi\rangle
       +\frac{1}{2}X(\xi,\xi)+o(|\xi|^2+|\tau|).
\]
The parabolic subjet $\mathcal P^{2,-}u(x,t)$ is defined by the
reverse inequality. The closed jets
$\overline{\mathcal P}^{2,\pm}u(x,t)$ are limits of jets at nearby
points, with the points, function values, and jet components all
converging. We write $J^{2,\pm}$ and $\overline J^{2,\pm}$ for
the corresponding elliptic jets, defined by a full quadratic
expansion in all variables. In particular, time is one of the
elliptic variables when we apply the maximum principle for semicontinuous functions below.

For a continuous operator, the test-function inequalities are
equivalent to the jet inequalities and pass to closed jets. For
$\mathcal Q_\gamma$, we use this fact at nonzero gradients. At
zero gradients we use the continuous extension for $\gamma>1$,
the envelopes for $\gamma=1$, and admissible tests for
$0<\gamma<1$. Normal coordinates identify the spatial Hessian
block of a jet with the covariant Hessian at its base point.

All first-order differential inequalities in this paper are
viscosity supersolution inequalities: the asserted inequality
must hold for every $C^1$ function touching the unknown from
below.

\subsection{Jets at two independent times}

The endpoint times in our comparisons are independent. We
therefore apply the elliptic maximum principle for semicontinuous functions to all space-time
variables and then retain the parabolic components of the jets.

\begin{lemma}\label{lemma:2.1 two-point-jets}
Let $\Omega\subset\R^n$ be open and let
$v\in C(\Omega\times(0,T))$. Suppose
\[
 Z(x,t,y,s)=v(x,t)-v(y,s)+\psi(x,t,y,s)
\]
has a local minimum at $P=(x_0,t_0,y_0,s_0)$, where $\psi$ is
$C^2$ near $P$. For every $\rho>0$, there are
$X,Y\in\Sym(n)$ such that
\begin{align}\label{eq: 2.1}
 (-\psi_t,-D_x\psi,X)
 &\in\overline{\mathcal P}^{2,-}v(x_0,t_0),\notag\\
 (\psi_s,D_y\psi,Y)
 &\in\overline{\mathcal P}^{2,+}v(y_0,s_0),
\end{align}
and
\begin{equation}\label{eq:2.2}
 \begin{pmatrix}-X&0\\0&Y\end{pmatrix}
 \leq D^2_{x,y}\psi+\rho I_{2n}.
\end{equation}
All derivatives of $\psi$ are evaluated at $P$.
\end{lemma}

\begin{proof}
Set $z=(x,t)$ and $w=(y,s)$. The function
\[
 -v(z)+v(w)-\psi(z,w)
\]
has a local maximum at $P$. Let
$H=D^2_{z,w}\psi(P)$ be the full space-time Hessian and choose
\[
 \alpha=\frac{\rho}{1+\|H\|^2},
\]
where $\|H\|$ is the operator norm. By the maximum principle for
semicontinuous functions~\cite[Theorem~3.2]{CIL}, there are
$S_1,S_2\in\Sym(n+1)$ with
\[
 (D_z\psi,S_1)\in\overline J^{2,+}(-v)(x_0,t_0),
\]
and
\[
 (D_w\psi,S_2)\in\overline J^{2,+}v(y_0,s_0),
\]
such that
\[
 \begin{pmatrix}S_1&0\\0&S_2\end{pmatrix}
 \leq H+\alpha H^2\leq H+\rho I_{2n+2}.
\]
Taking $X=-(S_1)_{xx}$ and $Y=(S_2)_{yy}$ and restricting to
vectors whose time components vanish proves \eqref{eq:2.2}.

Changing the sign of the first jet gives a closed elliptic subjet
of $v$. A full space-time jet projects to a parabolic jet with the
same time derivative and spatial Hessian block, since
\[
 |\Delta x|\,|\Delta t|+|\Delta t|^2
 =o(|\Delta x|^2+|\Delta t|).
\]
The full quadratic remainder has the same property. Projection
also commutes with the convergent sequences defining closed
jets, which proves \eqref{eq: 2.1}.
\end{proof}

The proof applies without change to two different continuous
functions at the two endpoints. On a manifold, we use separate
normal coordinate charts at the contact points. The matrix
parameter is chosen after the contact point and the full Hessian
of the penalty have been fixed.

\subsection{Changes of variables and singular regularity}
\label{subsec:regularity}

Throughout the paper, positivity of a function in
$C(M\times[0,T))$ includes positivity at time zero. On a closed
manifold this gives a positive initial minimum. When $T=\infty$,
all maximum arguments are first made below a finite terminal
time $b<T$.

On a closed manifold, if $u\in C(M\times[0,T))$ solves
$u_t=\mathcal Q_\gamma[u]$, then
\[
 \min_Mu(\cdot,0)\leq u(x,t)\leq\max_Mu(\cdot,0).
\]
To prove the upper bound, suppose it fails before a finite time
$b<T$. For sufficiently small $\varepsilon>0$, the function
\[
 u(x,t)-\max_Mu(\cdot,0)-\frac{\varepsilon}{b-t}
\]
has a positive maximum at an interior time. The resulting upper
test is spatially constant and has positive time derivative,
contradicting each of the critical-point conventions above.
The lower bound follows by using
$\min_Mu(\cdot,0)-\frac{\varepsilon}{b-t}$ as a lower barrier.
Thus a positive solution stays in a compact subinterval of
$(0,\infty)$.

For $\gamma=1$, the substitution $v=\log u$ gives
\[
 \mathcal Q_1[u]=u\mathcal Q_1[v]+Au|\nabla v|^2,
\]
and
\[
 v_t=\mathcal Q_1[v]+A|\nabla v|^2.
\]
For $\gamma\ne1$, set
\[
 \theta=\frac{\gamma-1}{\gamma},
 \quad v=\frac{u^\theta}{\theta}.
\]
Then $\frac{dv}{du}=u^{-\frac{1}{\gamma}}$ and the chain rule gives
\[
 \mathcal Q_\gamma[u]
 =u\mathcal Q_\gamma[v]
   +\frac{A}{\gamma}u^{\frac{1}{\gamma}}
                     |\nabla v|^{\gamma+1},
\]
and
\[
 v_t=\theta v\mathcal Q_\gamma[v]
             +\frac{A}{\gamma}|\nabla v|^{\gamma+1}.
\]
In particular, $\theta v=u^\theta>0$ even when $v<0$.
These identities hold in the viscosity sense: compose a test for
$v$ with the increasing inverse transformation to obtain a test
for $u$. On the compact positive range just obtained, a smooth
increasing transformation preserves singular admissibility. Its
spatial remainder is bounded by a constant multiple of $\omega$,
and the additional time remainder is $o(|t-t_0|)$. Multiplication
by a smooth positive function of time is handled in the same way,
with its time derivative included in the transformed equation.

The next lemma provides the regularity needed to localize the
pressure comparison near the initial-time corner. No sign condition
on the curvature is required.

\begin{lemma}[Regularity in the singular range]
\label{lemma:singular-regularity}
Let $(M,g,e^{-f}d\mu_g)$ be closed and connected, and let
$0<\gamma<1$, $A>0$, and $B\geq0$. Comparison holds for
subsolutions and supersolutions of $u_t=\mathcal Q_\gamma[u]$
that are continuous on $M\times[0,T)$: an ordering at time zero
persists. Every positive solution $u\in C(M\times[0,T))$ and
its pressure $v=\frac{u^\theta}{\theta}$, with
$\theta=\frac{\gamma-1}{\gamma}$, are Lipschitz in space and
time on $M\times[a,b]$, for $0<a<b<T$.
\end{lemma}

\begin{proof}
Choose $K\geq0$ such that $\Ric_f\geq-Kg$, which is possible
by compactness. We first prove comparison. Let $u_1$ be a
subsolution and $u_2$ a supersolution with
$u_1(\cdot,0)\leq u_2(\cdot,0)$. If this ordering fails, choose
a point $(x_*,t_*)$ where $u_1-u_2>0$ and a finite time
$t_*<b<T$. Fix $\eta>0$ sufficiently small that
\[
 u_1(x_*,t_*)-u_2(x_*,t_*)-\chi(t_*)>0,
 \quad
 \chi(t)=\eta t+\frac{\eta}{b-t}.
\]
Thus $\chi'\geq\eta$ and $\chi(t)\to\infty$ as $t\uparrow b$.
For $q>1+\frac{1}{\gamma}$, maximize
\[
 u_1(x,t)-u_2(y,s)-\chi(t)-\lambda d(x,y)^q
                         -\frac{(t-s)^2}{2\varepsilon}
\]
on $M\times[0,b)\times M\times[0,b]$. Its maximum is bounded
below by a fixed positive number, independently of $\lambda$
and $\varepsilon$. Boundedness of the two functions keeps $t$
uniformly away from $b$ and gives
\[
 d(x,y)\leq C\lambda^{-\frac{1}{q}},
\]
and
\[
 |t-s|\leq C\sqrt{\varepsilon}.
\]
Initial ordering and uniform continuity exclude $t=0$ and $s=0$
when $\lambda$ is sufficiently large and $\varepsilon$ is
sufficiently small. The terminal bound on $t$ then excludes $s=b$.
All contact points are therefore interior.

Suppose first that $d=d(x,y)>0$. Apply the full space-time
maximum principle for semicontinuous functions to $u_1$ and $-u_2$. The two spatial gradients
are parallel translates of the same vector of length
$p_0=\lambda qd^{q-1}$, and the temporal components satisfy
\[
 a_1-a_2=\chi'(t)\geq\eta.
\]
Along a unit-speed minimizing geodesic from $x$ to $y$, contract
the spatial matrix inequality against parallel endpoint vectors.
The longitudinal second variation is zero. The transverse index
forms, together with the difference of the drift terms, give
\begin{align*}
 \eta
 &\leq\mathcal Q_\gamma(x,q_x,X)
             -\mathcal Q_\gamma(y,q_y,Y)\\
 &\leq-Bp_0^\gamma\int_0^d\Ric_f(e_n,e_n)\,d\ell\\
 &\leq BKq^\gamma(\lambda d^q)^\gamma d^{1-\gamma}.
\end{align*}
Here $e_n$ is the unit tangent. At cut points we use the smooth
upper support for distance furnished by the lengths of geodesic
variations, as constructed in Section~\ref{sec:homogeneous}.
The matrix error is sent to zero at each fixed contact, where
$p_0>0$, before either doubling parameter is changed.

The quantities $\lambda d^q$ are uniformly bounded, while
$d\leq C\lambda^{-\frac{1}{q}}$. Since $\gamma<1$, the last
upper bound is less than $\eta$ when $\lambda$ is sufficiently
large, uniformly for small $\varepsilon$. This is a contradiction.
If a maximizing pair lies on the spatial diagonal, freeze the
opposite endpoint in the doubled expression. The resulting
spatial tests are constant multiples of $d^q$ and hence are
admissible. They give $a_1\leq0$ and $a_2\geq0$, again
contradicting $a_1-a_2\geq\eta$. This proves comparison. We choose $\lambda$ large and
$\varepsilon$ small to keep the contacts interior and make the
curvature error less than $\eta$; only the matrix error is sent
to zero at the resulting fixed contact.

We next obtain a time Lipschitz bound. Put
$M_0=\max_Mu(\cdot,0)$. The amplitude bound gives
$0<u\leq M_0$. Homogeneity implies that
\[
 u_\lambda(x,t)=\lambda u(x,\lambda^{\gamma-1}t)
\]
is a solution for $\lambda>1$. Since
$u_\lambda(\cdot,0)\geq u(\cdot,0)$, comparison yields
$u\leq u_\lambda$. The function $M_0-u$ is also a solution:
the operator is independent of the value of the unknown and
\[
 \mathcal Q_\gamma(x,-q,-X)
 =-\mathcal Q_\gamma(x,q,X).
\]
Changing sign exchanges upper and lower tests and preserves the
admissible remainder. Comparison applies to $M_0-u$ although
this function may vanish.

For $0<t_1<t_2<T$, take
$\lambda=(\frac{t_2}{t_1})^{\frac{1}{1-\gamma}}$ and apply
the preceding two comparisons at $t_2$. They give
\[
 -(\lambda-1)(M_0-u(x,t_1))
 \leq u(x,t_2)-u(x,t_1)
 \leq(\lambda-1)u(x,t_1).
\]
On a fixed strip $[a,b]$, the derivative of
$r\mapsto r^{\frac{1}{1-\gamma}}$ is bounded for
$1\leq r\leq\frac{b}{a}$. The mean value theorem therefore
bounds $\lambda-1$ by a constant multiple of $t_2-t_1$.
Since $0\leq u,M_0-u\leq M_0$, this proves a time Lipschitz
estimate uniform in $x$.

Fix an interior time $t_0$ and a time Lipschitz constant $H$ on
a neighborhood of $t_0$. We claim that the slice
$U=u(\cdot,t_0)$ satisfies
\[
 \mathcal Q_\gamma[U]\geq-H\quad\text{at upper spatial tests},
\]
and
\[
 \mathcal Q_\gamma[U]\leq H\quad\text{at lower spatial tests},
\]
provided the test gradient is nonzero. To see this, make an upper
spatial test $\varphi$ strict at $x_0$, without changing its first
two derivatives there, and choose a small closed coordinate ball
on which $U-\varphi$ has its unique maximum at $x_0$. For $H'>H$,
maximize
\[
 u(x,t)-\varphi(x)
           -H'\sqrt{(t-t_0)^2+\varepsilon^2}
\]
on that ball and a small closed time interval about $t_0$.
If $(x_\varepsilon,t_\varepsilon)$ is a maximizing point, the
time Lipschitz bound and comparison with $(x_0,t_0)$ give
\[
 (H'-H)|t_\varepsilon-t_0|\leq H'\varepsilon,
\]
and
\[
 (U-\varphi)(x_0)-(U-\varphi)(x_\varepsilon)
 \leq H'\varepsilon.
\]
Thus the contacts are interior for small $\varepsilon$ and
converge to $(x_0,t_0)$. The time derivative of the penalty has
absolute value at most $H'$. The subsolution inequality therefore
gives $\mathcal Q_\gamma[\varphi]\geq-H'$ in the limit.
Letting $H'\downarrow H$ proves the upper-test assertion.
For a lower test, minimize after adding the same time penalty;
the supersolution inequality gives the other assertion. These
inequalities also pass to closed spatial jets with nonzero
gradient.

It remains to bound the spatial oscillation of $U$. Write
\[
 D=\operatorname{diam}M,
 \quad \sigma=\frac{1}{D}+\frac{BKD}{2A},
 \quad \varphi(s)=\frac{L}{\sigma}(1-e^{-\sigma s}),
\]
where $L>0$ will be chosen below. If
$U(x)-U(y)-2\varphi(\frac{d(x,y)}{2})$ has a positive maximum,
then its contact distance is nonzero. Apply the maximum principle for semicontinuous functions
and use opposite longitudinal endpoint vectors and parallel
transverse endpoint vectors along a minimizing geodesic.
The contact gradients have length $\varphi'(\frac{d}{2})>0$,
so the slice inequalities apply. The longitudinal second
variation of the cost is $2\varphi''(\frac{d}{2})$, while the
transverse terms combine with the drift into $\Ric_f$. After
removing the matrix error, we obtain
\begin{align*}
 -2H
 &\leq 2A\left(\varphi'\left(\frac{d}{2}\right)\right)^{\gamma-1}
                 \varphi''\left(\frac{d}{2}\right)
   -B\left(\varphi'\left(\frac{d}{2}\right)\right)^\gamma
                   \int_0^d\Ric_f(e_n,e_n)\,d\ell\\
 &\leq(-2A\sigma+BKd)L^\gamma e^{-\frac{\gamma\sigma d}{2}}\\
 &\leq-\frac{2A}{D}L^\gamma e^{-\frac{\gamma\sigma D}{2}}.
\end{align*}
As before, a smooth upper distance support justifies the
calculation at cut points. Choosing
$L^\gamma>\frac{HD}{A}e^{\frac{\gamma\sigma D}{2}}$ gives a
contradiction. Reversing $x$ and $y$ and using
$\varphi(s)\leq Ls$ proves
\[
 |U(x)-U(y)|\leq Ld(x,y).
\]
The constants are uniform on every compact positive-time strip.
Finally, $\frac{dv}{du}=u^{-\frac{1}{\gamma}}$ is bounded on the
positive range of $u$, so the pressure has the same spatial and
temporal Lipschitz regularity.
\end{proof}

\section{Homogeneous uniformly elliptic equations on flat spaces}
\label{sec:fully}

We first prove Theorem~\ref{thm: main fully nonlinear}. The logarithmic
transformation allows us to compare two scaled Hessians by uniform
ellipticity. On Euclidean space, we also need a spatial localization,
since the logarithm of a positive bounded solution need not be bounded
from below.

\begin{proof}[Proof of Theorem~\ref{thm: main fully nonlinear}]
Set $v=\log u$ and $C=\frac{n\Lambda}{2\lambda}$. By exponentiating
upper and lower test functions and using the homogeneity of $F$, we
find that $v$ satisfies
\begin{equation}\label{eq: 3.1}
 v_t=F(D^2v+Dv\otimes Dv)
\end{equation}
in the viscosity sense.

We begin with a flat torus. Fix $0<a<b<T$ and consider
\[
 Z(x,t,y,s)=v(x,t)-v(y,s)+C\log\frac{t-a}{s-a}
                  +\frac{d(x,y)^2}{4\lambda(t-s)}
\]
for $a<s<t<b$. Suppose that $Z$ is negative at some pair. For a
sufficiently small fixed $\varepsilon>0$, the function
\[
 Z_\varepsilon(x,t,y,s)=Z(x,t,y,s)+\frac{\varepsilon}{b-t}
\]
is still negative there. We claim that it attains a negative interior
minimum. The function $v$ is bounded and uniformly continuous on
$M\times[a,b]$. The terminal term excludes $t\uparrow b$, and the
logarithmic term excludes $s\downarrow a$ whenever $t-s$ stays
positive. If $t-s\to0$ along a fixed negative sublevel, boundedness
forces $d(x,y)\to0$. Uniform continuity then gives
$v(x,t)-v(y,s)\to0$, which is impossible because all three penalty
terms are nonnegative. These observations exclude the boundary of
the time triangle, including the corner $s=t=a$, and compactness
gives the claimed minimum.

Let $(x_0,t_0,y_0,s_0)$ be a point of negative minimum. Write the
torus as the quotient of $\R^n$ by a lattice $\Gamma$. We choose
local lifts $\widetilde x,\widetilde y$ and $k_0\in\Gamma$ such that
\[
 d(x_0,y_0)^2=|\widetilde x_0-\widetilde y_0+k_0|^2.
\]
For nearby endpoints, the expression on the right with $x,y$ in
place of $x_0,y_0$ is a smooth upper support for $d(x,y)^2$.
Consequently, replacing the squared distance by this expression
preserves the local minimum, even if the endpoints are cut points.
With
\[
 \psi(x,t,y,s)=C\log\frac{t-a}{s-a}
       +\frac{|\widetilde x-\widetilde y+k_0|^2}{4\lambda(t-s)}
       +\frac{\varepsilon}{b-t},
\]
the function $v(x,t)-v(y,s)+\psi(x,t,y,s)$ therefore has a local
minimum at $(x_0,t_0,y_0,s_0)$.

Lemma~\ref{lemma:2.1 two-point-jets} gives, for every $\rho>0$,
matrices $X,Y\in\Sym(n)$ satisfying
\begin{align}
 (-\psi_t,-D_x\psi,X)
   &\in\overline{\mathcal P}^{2,-}v(x_0,t_0),\notag\\
 (\psi_s,D_y\psi,Y)
   &\in\overline{\mathcal P}^{2,+}v(y_0,s_0),
 \label{eq: 3.2}
\end{align}
and
\begin{equation}\label{eq: 3.3}
 \begin{pmatrix}X&0\\0&-Y\end{pmatrix}
       \geq-D^2_{x,y}\psi-\rho I_{2n}.
\end{equation}
All derivatives are evaluated at the minimum. We now suppress the
subscripts on its coordinates and write
\[
 p=-D_x\psi=D_y\psi
       =-\frac{\widetilde x-\widetilde y+k_0}{2\lambda(t-s)}.
\]
The spatial Hessian of the penalty is
\begin{equation}\label{eq: 3.4}
 D^2_{x,y}\psi=\frac{1}{2\lambda(t-s)}
       \begin{pmatrix}I_n&-I_n\\-I_n&I_n\end{pmatrix}.
\end{equation}
Since $F$ is continuous, the viscosity inequalities hold for the
closed jets in \eqref{eq: 3.2}. Hence
\[
 -\psi_t\geq F(X+p\otimes p),
\]
and
\[
 \psi_s\leq F(Y+p\otimes p).
\]
Multiplying by the squared time weights and subtracting gives
\begin{align}
  -(t-a)^2\psi_t-(s-a)^2\psi_s
 &\geq(t-a)^2F(X+p\otimes p)\notag\\
 &\quad-(s-a)^2F(Y+p\otimes p).
 \label{eq: 3.5}
\end{align}

We apply ellipticity after scaling the two matrices. Choose
$A\in\Sym(n)$, with $\lambda I_n\leq A\leq\Lambda I_n$, attaining
the lower Pucci minimum for
\[
 (t-a)^2(X+p\otimes p)-(s-a)^2(Y+p\otimes p).
\]
Positive homogeneity then gives
\begin{align}
 &(t-a)^2F(X+p\otimes p)-(s-a)^2F(Y+p\otimes p)\notag\\
 &\quad\geq(t-a)^2\tr(AX)-(s-a)^2\tr(AY)
       +\bigl((t-a)^2-(s-a)^2\bigr)\langle Ap,p\rangle\notag\\
 &\quad\geq(t-a)^2\tr(AX)-(s-a)^2\tr(AY)
       +\lambda\bigl((t-a)^2-(s-a)^2\bigr)|p|^2.
 \label{eq: 3.6}
\end{align}
The last inequality uses $t>s>a$. This choice of $A$ uses only
uniform ellipticity; no derivative of $F$ is involved.

Contract \eqref{eq: 3.3} against
$((t-a)A^{\frac{1}{2}}e_i,(s-a)A^{\frac{1}{2}}e_i)$ and sum over
an orthonormal basis $e_1,\ldots,e_n$. By \eqref{eq: 3.4}, we obtain
\begin{align}
 &(t-a)^2\tr(AX)-(s-a)^2\tr(AY)\notag\\
 &\quad\geq-\frac{t-s}{2\lambda}\tr A
       -\rho\bigl((t-a)^2+(s-a)^2\bigr)\tr A\notag\\
 &\quad\geq-C(t-s)
       -n\Lambda\rho\bigl((t-a)^2+(s-a)^2\bigr).
 \label{eq: 3.7}
\end{align}
On the other hand,
\[
 \psi_t=\frac{C}{t-a}-\lambda|p|^2
                         +\frac{\varepsilon}{(b-t)^2},
\]
and
\[
 \psi_s=-\frac{C}{s-a}+\lambda|p|^2.
\]
Substituting these identities into \eqref{eq: 3.5} and applying
\eqref{eq: 3.6} and \eqref{eq: 3.7}, we see that both the terms
containing $C$ and those containing $|p|^2$ cancel. What remains is
\[
 0<\frac{\varepsilon(t-a)^2}{(b-t)^2}
       \leq n\Lambda\rho\bigl((t-a)^2+(s-a)^2\bigr).
\]
The contact point and $\varepsilon$ are fixed independently of
$\rho$. Letting $\rho\downarrow0$ gives a contradiction. We have
proved $Z\geq0$ on the torus. For a fixed pair with $0<s<t<T$, we
choose $b>t$ and let $a\downarrow0$. Exponentiation proves
\eqref{eq 1.2 harnack}.

We next consider $\R^n$. Fix $\zeta>0$ and replace $v$ by
$\log(u+\zeta)$. Addition of a constant preserves the original
equation, so this function still satisfies \eqref{eq: 3.1}. It is
also bounded and uniformly continuous on every strip
$\R^n\times[a,b]$. Define $Z$ as above with $d(x,y)=|x-y|$.
If $Z$ is negative at some pair, choose a small fixed
$\varepsilon>0$ and consider
\[
 Z_\varepsilon(x,t,y,s)+\eta(|x|^2+|y|^2)
\]
for $\eta>0$. For all sufficiently small $\eta$, this function
takes a value below a fixed negative number. The spatial penalty
is coercive, and the preceding time-boundary argument applies.
It therefore attains a negative interior minimum.

At these minima we have
\[
 \eta(|x|^2+|y|^2)+\frac{|x-y|^2}{4\lambda(t-s)}
       \leq2\|v\|_{L^\infty(\R^n\times[a,b])}.
\]
Moreover, $t-s$ has a positive lower bound independent of small
$\eta$. Otherwise the displayed inequality forces $|x-y|\to0$,
and uniform continuity contradicts the fixed negative bound for
the minimum values. Thus $p=-\frac{x-y}{2\lambda(t-s)}$ remains
bounded, and
\[
 \eta(|x|+|y|)\longrightarrow0,
\]
and
\[
 \eta^2(|x|^2+|y|^2)\longrightarrow0
\]
as $\eta\downarrow0$.

We apply the jet lemma to $\psi+\eta(|x|^2+|y|^2)$. The time
derivatives are unchanged, the endpoint gradients become
$p-2\eta x$ and $p+2\eta y$, and \eqref{eq: 3.3} holds with
$\rho+2\eta$ in place of $\rho$. Choose the Pucci minimizing
matrix for the difference of the corresponding scaled endpoint
matrices. The same contraction applies. The additional
gradient terms are controlled by
\[
 \left|\langle A(p-2\eta x),p-2\eta x\rangle
                     -\langle Ap,p\rangle\right|
       \leq4\Lambda\bigl(\eta|p||x|+\eta^2|x|^2\bigr),
\]
and the analogous bound with $y$ at the other endpoint. Repeating
the cancellation above gives
\begin{align*}
 \frac{\varepsilon(t-a)^2}{(b-t)^2}
 &\leq n\Lambda(\rho+2\eta)
                     \bigl((t-a)^2+(s-a)^2\bigr)\\
 &\quad+4\Lambda(b-a)^2
       \bigl(\eta|p|(|x|+|y|)+\eta^2(|x|^2+|y|^2)\bigr).
\end{align*}
For each fixed $\eta$, we first let $\rho\downarrow0$, and then
let $\eta\downarrow0$. The right-hand side tends to zero. The
left-hand side has a positive lower bound, since
$t-a\geq t-s$ and $b-t\leq b-a$. This contradiction proves
$Z\geq0$ for $v=\log(u+\zeta)$. Letting $a\downarrow0$ and then
$\zeta\downarrow0$ proves \eqref{eq 1.2 harnack} on $\R^n$.

Finally, let a $C^1$ function $\varphi$ touch $v=\log u$ from
below at $(x,t)$. In local flat coordinates, choose
\[
 y=x+2\lambda(t-s)D\varphi(x,t)
\]
for $s<t$ sufficiently close to $t$. The logarithmic form of the
Harnack inequality and the lower touching property imply
\[
 0\leq\varphi(x,t)-\varphi(y,s)+C\log\frac{t}{s}
                         +\lambda(t-s)|D\varphi(x,t)|^2.
\]
Dividing by $t-s$ and letting $s\uparrow t$, we obtain
\[
 \varphi_t(x,t)-\lambda|D\varphi(x,t)|^2+\frac{C}{t}\geq0.
\]
This proves \eqref{eq 1.3 differential Harnack} in the first-order
viscosity sense.
\end{proof}

\section{Normalized weighted \texorpdfstring{$p$}{p}-Laplacian and infinity heat flows}
\label{sec:homogeneous}

We prove Proposition~\ref{intro:prop-degree-one} directly. Theorems~\ref{thm:intro-normalized-p}
and~\ref{thm:intro-infinity} then follow by choosing the longitudinal
and transverse coefficients. This also makes clear why the infinity
heat estimate requires no curvature assumption.

\begin{proof}[Proof of Proposition~\ref{intro:prop-degree-one}]
Set $v=\log u$. The change of variables in
Section~\ref{sec:prelim} gives
\begin{equation}\label{np:log-equation}
 v_t=\mathcal Q_1[v]+A|\nabla v|^2
\end{equation}
in the viscosity sense. We retain
$\eta_N=\frac{1}{2}+\frac{(N-1)B}{2A}$ when $B>0$ and set
$\eta_N=\frac{1}{2}$ when $B=0$. We will prove
\begin{equation}\label{np:log-harnack}
 v(x,t)-v(y,s)+\eta_N\log\frac{t}{s}
                    +\frac{d(x,y)^2}{4A(t-s)}\geq0.
\end{equation}

For $\varepsilon>0$, define
\[
 h_\varepsilon(r)=\int_0^r\frac{w^2}{w+\varepsilon^2}\,dw.
\]
These functions increase to $\frac{r^2}{2}$ as
$\varepsilon\downarrow0$, and
\[
 h_\varepsilon'(r)=\frac{r^2}{r+\varepsilon^2},
\]
and
\[
 h_\varepsilon''(r)=1-\frac{\varepsilon^4}{(r+\varepsilon^2)^2}.
\]
Thus $0\leq h_\varepsilon''\leq1$,
$0<h_\varepsilon'(r)\leq r$ for $r>0$, and
$h_\varepsilon(0)=h_\varepsilon'(0)=h_\varepsilon''(0)=0$.
In particular, $h_\varepsilon(d(x,y))$ is $C^2$ near the
diagonal, with zero spatial gradient and Hessian there. We also have
\begin{equation}\label{np:kernel-error}
 0\leq2h_\varepsilon(r)-h_\varepsilon'(r)^2
                         \leq2\varepsilon^2r.
\end{equation}
The lower bound follows by differentiation, since
$2h_\varepsilon'(1-h_\varepsilon'')\geq0$. For the upper bound,
use $2h_\varepsilon(r)\leq r^2$ and
$0\leq r-h_\varepsilon'(r)\leq\varepsilon^2$.

Fix $0<b<T$. If \eqref{np:log-harnack} fails at a pair with
$0<s<t<b$, then, for all sufficiently small $\varepsilon$,
\begin{equation}\label{np:Z}
 Z_\varepsilon(x,t,y,s)
 =v(x,t)-v(y,s)+\eta_N\log\frac{t}{s}
 +\frac{h_\varepsilon(d(x,y))}{2A(t-s)}
 +\frac{\varepsilon}{b-t}
\end{equation}
takes a value below a fixed negative number. We show that it
attains an interior minimum, with $s$ and $t-s$ bounded below
independently of small $\varepsilon$.

The function $v$ is bounded and uniformly continuous on
$M\times[0,b]$, and every added term in \eqref{np:Z} is
nonnegative. On a fixed negative sublevel, $t-s\to0$ would
force $d(x,y)\to0$: for small $\varepsilon$, the kernels have
a common positive lower bound at every fixed positive distance.
Uniform continuity would then give
$v(x,t)-v(y,s)\to0$, a contradiction. Once $t-s$ is bounded
below, the logarithmic term excludes $s\downarrow0$. For each
fixed $\varepsilon$, the terminal term excludes $t\uparrow b$.
Compactness gives the asserted minimum and the uniform lower
bounds for $s$ and $t-s$.

Let $(x_0,t_0,y_0,s_0)$ be such a minimum. If $x_0=y_0$,
freezing the earlier endpoint gives the lower test
\[
 \varphi(x,t)=v(x_0,t_0)+\eta_N\log\frac{t_0}{t}
 -\frac{h_\varepsilon(d(x,x_0))}{2A(t-s_0)}
 +\frac{\varepsilon}{b-t_0}-\frac{\varepsilon}{b-t}
\]
at $(x_0,t_0)$. Its spatial gradient and Hessian vanish there,
so both critical-point envelopes in \eqref{np:log-equation}
vanish. But
\[
 \varphi_t(x_0,t_0)
 =-\frac{\eta_N}{t_0}-\frac{\varepsilon}{(b-t_0)^2}<0,
\]
contrary to the supersolution inequality. Hence $x_0\ne y_0$.

Choose a constant-speed minimizing geodesic
$\alpha:[s_0,t_0]\to M$ from $y_0$ to $x_0$ and a parallel
orthonormal frame $e_1,\ldots,e_n$, with
$e_n=\frac{\alpha'}{|\alpha'|}$. We use the variations
\[
 \alpha_i(r,\tau)
 =\exp_{\alpha(\tau)}\bigl(r\tau e_i(\tau)\bigr)
\]
for $s_0\leq\tau\leq t_0$ and $1\leq i\leq n$.
Their endpoint velocities are $s_0e_i(s_0)$ and
$t_0e_i(t_0)$. This is the weighted variation used
in~\cite{AndrewsClutterbuck2013,LiWang2017,LiTuWang2024},
with the times as weights.

One smooth upper support for distance realizes all these
variations. Let $P_{r,\tau}$ denote parallel transport along
$\alpha$. For small $U\in T_{y_0}M$ and $V\in T_{x_0}M$,
form the curve
\[
 \Gamma_{U,V}(\tau)=\exp_{\alpha(\tau)}\left(
 \frac{t_0-\tau}{t_0-s_0}P_{s_0,\tau}U
 +\frac{\tau-s_0}{t_0-s_0}P_{t_0,\tau}V\right).
\]
Its length defines a smooth function
$\widetilde d(\exp_{x_0}V,\exp_{y_0}U)$ near $(x_0,y_0)$,
because its speed remains nonzero for small displacements.
It majorizes distance and agrees with it at the contact.
Taking $U=rs_0e_i(s_0)$ and $V=rt_0e_i(t_0)$ gives exactly
$\alpha_i(r,\cdot)$. The same support therefore realizes all
the required variations, including at cut points. We keep it,
and its construction times, fixed when differentiating the
penalty in its time variables.

Let $\psi$ be the last three terms in \eqref{np:Z}, with
$d$ replaced by $\widetilde d$. Since $h_\varepsilon$ is
increasing, the contact remains a local minimum. In normal
coordinates based on the parallel endpoint frames,
Lemma~\ref{lemma:2.1 two-point-jets} gives
\begin{align}
 (-\psi_t,-D_x\psi,X)
 &\in\overline{\mathcal P}^{2,-}v(x_0,t_0),\notag\\
 (\psi_s,D_y\psi,Y)
 &\in\overline{\mathcal P}^{2,+}v(y_0,s_0),\notag\\
 \begin{pmatrix}-X&0\\0&Y\end{pmatrix}
 &\leq D^2_{x,y}\psi+\rho I_{2n}
 \label{np:jets}
\end{align}
for every $\rho>0$.

Write $t=t_0$, $s=s_0$, and $d=d(x_0,y_0)>0$. At the
contact, the endpoint gradients are
\[
 -D_x\psi=-\frac{h_\varepsilon'(d)}{2A(t-s)}e_n(t),
\]
and
\[
 D_y\psi=-\frac{h_\varepsilon'(d)}{2A(t-s)}e_n(s).
\]
They are nonzero, so the operator is continuous at these jets
and its viscosity inequalities pass to the closed jets.
Write
\[
 \mathcal L X=AX_{nn}+B\sum_{i<n}X_{ii},
\]
and define $\mathcal L Y$ in the same way. Differentiating the
time terms in the penalty gives
\begin{align}
 \mathcal L X+\frac{Bh_\varepsilon'(d)}{2Ad}(f\circ\alpha)'(t)
 &\leq-\frac{\eta_N}{t}-\frac{\varepsilon}{(b-t)^2}
 +\frac{2h_\varepsilon(d)-h_\varepsilon'(d)^2}{4A(t-s)^2},\notag\\
 \mathcal L Y+\frac{Bh_\varepsilon'(d)}{2Ad}(f\circ\alpha)'(s)
 &\geq-\frac{\eta_N}{s}
 +\frac{2h_\varepsilon(d)-h_\varepsilon'(d)^2}{4A(t-s)^2}.
 \label{np:endpoint-jets}
\end{align}
The drift signs are positive because both gradients point
along $-e_n$, while $\alpha'=\frac{d}{t-s}e_n$.

The longitudinal variation has length $d+r(t-s)$ for small
$r$. For $i<n$, the first length variation is zero. Since
the endpoint curves are geodesics, the second variation
formula gives
\[
 \sum_{i<n}\left.\frac{d^2}{dr^2}
 \operatorname{Length}(\alpha_i(r,\cdot))\right|_{r=0}
 =\frac{t-s}{d}\int_s^t
 \bigl((n-1)-\tau^2\Ric(\alpha',\alpha')\bigr)\,d\tau.
\]
Contract the last inequality in \eqref{np:jets} against
$(te_i(t),se_i(s))$, with the nonnegative weights $A$ for
$i=n$ and $B$ for $i<n$. Choose $\rho$ after the contact
is fixed so that
\[
 \rho(t^2+s^2)\bigl(A+(n-1)B\bigr)\leq\varepsilon^2.
\]
Then
\begin{align}
 t^2\mathcal L X-s^2\mathcal L Y
 &\geq-\frac{t-s}{2}h_\varepsilon''(d)-\varepsilon^2\notag\\
 &\quad-\frac{Bh_\varepsilon'(d)}{2Ad}
 \int_s^t\bigl((n-1)-\tau^2\Ric(\alpha',\alpha')\bigr)\,d\tau.
 \label{np:second-variation}
\end{align}

When $B>0$ and $N>n$, integration by parts and the weighted
Ricci bound give
\begin{align*}
 &\int_s^t\tau^2\Ric(\alpha',\alpha')\,d\tau
       +[\tau^2(f\circ\alpha)']_s^t\\
 &\quad=\int_s^t\left(
 \tau^2\Ric_f^N(\alpha',\alpha')
 +\frac{\bigl(\tau(f\circ\alpha)'+N-n\bigr)^2}{N-n}
 -(N-n)\right)\,d\tau\\
 &\quad\geq-(N-n)(t-s).
\end{align*}
When $B>0$ and $N=n$, the same bound follows because $f$ is constant
and $\Ric\geq0$. Combining it with
\eqref{np:second-variation}, and using
$h_\varepsilon''\leq1$ and $h_\varepsilon'(d)\leq d$, gives
\[
 t^2\mathcal L X-s^2\mathcal L Y
 +\frac{Bh_\varepsilon'(d)}{2Ad}
       [\tau^2(f\circ\alpha)']_s^t
 \geq-\eta_N(t-s)-\varepsilon^2.
\]
When $B=0$, this inequality follows directly from the
longitudinal term in \eqref{np:second-variation}, with
$\eta_N=\frac{1}{2}$. No curvature bound is used.

Multiplying the two inequalities in \eqref{np:endpoint-jets}
by $t^2$ and $s^2$, respectively, and subtracting gives an
upper bound for the same expression. The drift terms and
$-\eta_N(t-s)$ cancel. By \eqref{np:kernel-error},
\begin{align*}
 0&\leq-\frac{\varepsilon t^2}{(b-t)^2}
 +\frac{t+s}{4A(t-s)}
       \bigl(2h_\varepsilon(d)-h_\varepsilon'(d)^2\bigr)
 +\varepsilon^2\\
 &\leq-\frac{\varepsilon t^2}{(b-t)^2}
 +\varepsilon^2\left(1+\frac{d(t+s)}{2A(t-s)}\right).
\end{align*}
The coefficient of $\varepsilon^2$ is uniformly bounded,
since $d\leq\operatorname{diam}M$, $t<b$, and $t-s$ stays
away from zero. The coefficient $\frac{t^2}{(b-t)^2}$ has
a uniform positive lower bound. Division by $\varepsilon$
and passage to the limit give a contradiction. No lower
bound for the contact distance is needed: it is positive
for each $\varepsilon$, and the matrix error is chosen
afterward. Since $b<T$ was arbitrary,
\eqref{np:log-harnack} holds for all $0<s<t<T$.
Exponentiating proves the Harnack inequality in the proposition.

To obtain the differential inequality, let a $C^1$ function
$\varphi$ touch $v$ from below at $(x,t)$. Apply
\eqref{np:log-harnack} with
\[
 y=\exp_x\bigl(2A(t-s)\nabla\varphi(x,t)\bigr)
\]
for $s<t$ close to $t$. The lower contact permits us to
replace $v(y,s)$ by $\varphi(y,s)$ while keeping
$v(x,t)=\varphi(x,t)$. Dividing by $t-s$ and letting
$s\uparrow t$ gives
\[
 \varphi_t-A|\nabla\varphi|^2+\frac{\eta_N}{t}\geq0.
\]
This proves the first-order viscosity conclusion, including
at critical points of the test function.
\end{proof}

\begin{proof}[Proof of Theorem~\ref{thm:intro-normalized-p}]
Take $A=p-1$ and $B=1$ in
Proposition~\ref{intro:prop-degree-one}. Then
$\mathcal Q_1=\Delta_{p,f}^{\mathrm N}$ and
\[
 \eta_N=\frac{1}{2}+\frac{N-1}{2(p-1)}
       =\frac{N+p-2}{2(p-1)}.
\]
The proposition gives \eqref{intro:normalized-p-harnack}
and \eqref{intro:normalized-p-differential}. The
critical-point envelopes agree by the conventions in
Section~\ref{sec:prelim}.
\end{proof}

\begin{proof}[Proof of Theorem~\ref{thm:intro-infinity}]
Take $A=1$ and $B=0$. Then
$\mathcal Q_1=\Delta_\infty^{\mathrm N}$, with the same
critical-point envelopes, and $\eta_N=\frac{1}{2}$.
Proposition~\ref{intro:prop-degree-one} gives
\eqref{intro:infinity-harnack} and
\eqref{intro:infinity-differential} on every closed,
connected Riemannian manifold.
\end{proof}

\section{Homogeneous diffusion of degree \texorpdfstring{$\gamma\ne1$}{gamma not equal to one}}
\label{sec:homogeneous-power}

In this section we prove Theorem~\ref{intro:homogeneous-pressure}
and its specialization to the weighted $p$-Laplacian. The pressure
transformation replaces the logarithmic transformation used in
Section~\ref{sec:homogeneous}. The corresponding variations of the
minimizing geodesic have weights given by powers of time.

\begin{proof}[Proof of Theorem~\ref{intro:homogeneous-pressure}]
For this proof, set $r=1-(\gamma-1)\kappa$. The parameter
condition in the statement implies $r>0$, and we have
\begin{equation}\label{gneq1:cancellation}
 \frac{r}{\gamma+1}\left(1+\frac{(N-1)\gamma B}{A}\right)
 =\kappa,\quad r-1+\gamma\theta\kappa=0.
\end{equation}
Evaluating the integral in the definition of $\Lambda$ gives
\[
 \Lambda(s,t)=\left(
 \frac{\gamma r}{A(\gamma+1)(t^r-s^r)}
 \right)^{\frac{1}{\gamma}}.
\]
When $0<\gamma<1$, we first replace $u(x,t)$ by
$u(x,t+\tau_0)$ for a fixed $\tau_0>0$, retaining the notation
$u,v,T$ on the translated interval. By
Lemma~\ref{lemma:singular-regularity}, the translated pressure is
Lipschitz in space and time up to the new initial time on every
compact time interval. We remove this translation at the end of
the proof. No translation is needed when $\gamma>1$.

The change of variables in subsection~\ref{subsec:regularity} gives
\begin{equation}\label{gneq1:pressure}
 v_t=\theta v\mathcal Q_\gamma[v]
             +\frac{A}{\gamma}|\nabla v|^{\gamma+1},
 \quad \theta v=u^\theta>0.
\end{equation}
This equation holds in the viscosity sense, including the
admissible critical-point tests in the singular range. Set
$W(x,t)=t^{\theta\kappa}v(x,t)$. By
\eqref{gneq1:cancellation}, its equation is
\begin{equation}\label{gneq1:weighted-equation}
 W_t=\theta t^{r-1}W\mathcal Q_\gamma[W]
       +\frac{A}{\gamma}t^{r-1}|\nabla W|^{\gamma+1}
       +\frac{\theta\kappa}{t}W.
\end{equation}
The coefficient $\theta W$ is positive in both ranges of $\gamma$.

We regularize the distance cost by defining
\begin{equation}\label{gneq1:kernel}
 h_\varepsilon(\ell)
 =\int_0^\ell\frac{w^{\frac{\gamma+1}{\gamma}}}
                         {w+\varepsilon^2}\,dw.
\end{equation}
As $\varepsilon\downarrow0$, these functions increase to
$\frac{\gamma}{\gamma+1}\ell^{\frac{\gamma+1}{\gamma}}$
uniformly on bounded intervals. They satisfy
\begin{equation}\label{gneq1:kernel-bounds}
\begin{aligned}
 0&<(h_\varepsilon')^{\gamma-1}h_\varepsilon''
       \leq\frac{1}{\gamma},\\
 0&<(h_\varepsilon')^\gamma\leq\ell,\\
 0&\leq h_\varepsilon(\ell)
       -\frac{\gamma}{\gamma+1}h_\varepsilon'(\ell)^{\gamma+1}
       \leq\gamma\varepsilon^2\ell^{\frac{1}{\gamma}},
\end{aligned}
\end{equation}
for $\ell>0$. Indeed,
\[
 \frac{d}{d\ell}(h_\varepsilon'(\ell))^\gamma
 =\left(\frac{\ell}{\ell+\varepsilon^2}\right)^\gamma
     \left(1+\frac{\gamma\varepsilon^2}{\ell+\varepsilon^2}\right)
 \leq1.
\]
The last inequality follows by applying the increasing function
$x^\gamma(1+\gamma-\gamma x)$ on $[0,1]$ to
$x=\frac{\ell}{\ell+\varepsilon^2}$. It also shows that the
difference on the last line of \eqref{gneq1:kernel-bounds}
is increasing from zero. To obtain its upper bound, use
\[
 h_\varepsilon(\ell)
 \leq\frac{\gamma}{\gamma+1}\ell^{\frac{\gamma+1}{\gamma}},
\]
and
\[
 0\leq\ell-(h_\varepsilon')^\gamma\leq\gamma\varepsilon^2,
\]
together with the mean value theorem. For fixed $\varepsilon>0$,
$h_\varepsilon(\ell)=O(\ell^{2+\frac{1}{\gamma}})$ near zero.
Thus $h_\varepsilon(d)$ is $C^2$ near the diagonal, with zero
gradient and Hessian there. Moreover,
\[
 (h_\varepsilon')^{\gamma-1}
 \left(h_\varepsilon''+\frac{h_\varepsilon'}{\ell}\right)
 \longrightarrow0
\]
as $\ell\downarrow0$. The flat tests constructed from this kernel are therefore
admissible also when $0<\gamma<1$.

Fix a finite time $b<T$. Suppose that
\eqref{intro:general-pressure-harnack} fails for some
$0<s<t<b$. For all sufficiently small $\varepsilon>0$,
\begin{equation}\label{gneq1:Z}
 Z_\varepsilon(x,t,y,s)
 =W(x,t)-W(y,s)+\Lambda(s,t)h_\varepsilon(d(x,y))
                         +\frac{\varepsilon}{b-t}
\end{equation}
then has a value below a fixed negative number. We first show
that it attains such a minimum in the interior of the time
triangle, with $s$ and $t-s$ bounded below independently of
small $\varepsilon$.

For $\gamma>1$, $W>0$ extends continuously to zero at time zero.
Since $Z_\varepsilon\geq-W(y,s)$, a fixed negative sublevel
stays away from $s=0$. If $t-s\to0$ on this sublevel, the
boundedness of $W$ and $\Lambda(s,t)\to\infty$ force
$d(x,y)\to0$. Here the convergence of $h_\varepsilon$ is
uniform, so the assertion holds uniformly for small
$\varepsilon$. Uniform continuity of $W$ now contradicts
the fixed negative bound. For each fixed $\varepsilon$, the
terminal penalty excludes $t=b$, and compactness gives the
required interior minimum.

For $0<\gamma<1$, we must also control the joint initial
corner. Let $L$ be a common spatial and temporal Lipschitz
constant for $v$ on $M\times[0,b]$. Since
$\theta\kappa<0$ and $r>1$, we have
\[
 \left(\frac{s}{t}\right)^{\theta\kappa}-1
 \geq-\theta\kappa\frac{t-s}{t},
\]
and
\[
 \frac{\Lambda(s,t)}{t^{\theta\kappa}}
 \geq\left(\frac{\gamma}{A(\gamma+1)(t-s)}
           \right)^{\frac{1}{\gamma}}.
\]
For the second inequality we used
$t^r-s^r\leq rt^{r-1}(t-s)$ and
$r-1=-\gamma\theta\kappa$. The spatial and temporal
Lipschitz bounds give a contribution at least
$-Ld(x,y)-L(t-s)$. Minimizing the distance terms over
$d(x,y)\geq0$ gives
\begin{align}
 &t^{-\theta\kappa}\left(W(x,t)-W(y,s)
 +\frac{\gamma}{\gamma+1}\Lambda(s,t)
                   d(x,y)^{\frac{\gamma+1}{\gamma}}\right)
       \notag\\
 &\quad\geq(t-s)\left(
       \frac{-\theta\kappa\min(-v)}{t}
                    -L-\frac{A}{\gamma}L^{\gamma+1}\right).
       \label{gneq1:initial-corner}
\end{align}
The minimum of $-v$ is positive on $M\times[0,b]$.
Choose $0<t_*<b$ so that the right-hand side is nonnegative
for $0<s<t\leq t_*$, and minimize $Z_\varepsilon$ over
$t\geq t_*$. On this region $W(x,t)$ is bounded below,
whereas $-W(y,s)\to+\infty$ uniformly as $s\downarrow0$.
A fixed negative sublevel consequently stays away from $s=0$.
The preceding distance-cost argument then bounds $t-s$ away
from zero. These bounds make the regularized cost converge
uniformly on that sublevel. By \eqref{gneq1:initial-corner},
the face $t=t_*$ cannot contain it for small $\varepsilon$.
The terminal penalty and compactness give an interior
minimum with $t>t_*$, which is also a local minimum in the
full time triangle.

Write $(x_0,t_0,y_0,s_0)$ for the resulting minimum. It cannot
lie on the spatial diagonal. Otherwise, freezing the earlier
endpoint gives the lower test
\[
 \varphi(x,t)=W(x_0,t_0)-\Lambda(s_0,t)h_\varepsilon(d(x,x_0))
       +\frac{\varepsilon}{b-t_0}-\frac{\varepsilon}{b-t}
\]
at $(x_0,t_0)$. Its spatial gradient and Hessian vanish,
and the test is admissible in the singular range. Equation
\eqref{gneq1:weighted-equation} would therefore imply
\[
 -\frac{\varepsilon}{(b-t_0)^2}
 =\varphi_t(x_0,t_0)
 \geq\frac{\theta\kappa}{t_0}W(x_0,t_0)>0,
\]
a contradiction.

For the rest of the calculation, write $s=s_0$, $t=t_0$,
$d=d(x_0,y_0)>0$, and $\Lambda=\Lambda(s,t)$. Parametrize a minimizing
geodesic $\alpha:[s,t]\to M$ from $y_0$ to $x_0$ so that
\[
 |\alpha'(\tau)|=\frac{dr\tau^{r-1}}{t^r-s^r},
\]
and a parallel orthonormal frame $e_1,\ldots,e_n$ for which
$e_n=\frac{\alpha'}{|\alpha'|}$. We use the variations
\begin{equation}\label{gneq1:variations}
 \alpha_i(\ell,\tau)
 =\exp_{\alpha(\tau)}\bigl(\ell\tau^r e_i(\tau)\bigr),
\end{equation}
for $1\leq i\leq n$.

To treat cut points, we construct a single smooth upper
support for distance realizing all these variations. Divide
the normal-coordinate vectors of nearby endpoints by $s^r$
and $t^r$, parallel transport them along $\alpha$, and
interpolate affinely in $\tau$. Multiply the resulting vector
field by $\tau^r$, exponentiate along $\alpha$, and take the
length of this curve. For sufficiently small endpoint
displacements, its speed stays nonzero, so its length is
smooth. The length majorizes distance and equals $d$ at the
contact. On each pair of endpoint curves in
\eqref{gneq1:variations}, it equals the length of
$\alpha_i(\ell,\cdot)$. We keep this support fixed when
differentiating in the endpoint times.

Let $\psi$ be the sum of the last two terms of
\eqref{gneq1:Z}, with distance replaced by this upper
support. Since $h_\varepsilon$ is increasing, the local
minimum is preserved. Lemma~\ref{lemma:2.1 two-point-jets}
gives, for every $\rho>0$, matrices $X,Y$ such that
\begin{align}
 (-\psi_t,-D_x\psi,X)
       &\in\overline{\mathcal P}^{2,-}W(x_0,t),
       \label{gneq1:jets}\\
 (\psi_s,D_y\psi,Y)
       &\in\overline{\mathcal P}^{2,+}W(y_0,s),\notag\\
 \begin{pmatrix}-X&0\\0&Y\end{pmatrix}
       &\leq D^2_{x,y}\psi+\rho I_{2n}.\notag
\end{align}
All spatial derivatives are taken in normal coordinates at
the endpoints. Put $q=\Lambda h_\varepsilon'(d)>0$, and evaluate
all kernel derivatives below at $d$. The first variation gives
$-D_x\psi=-qe_n(t)$ and $D_y\psi=-qe_n(s)$. Thus the
viscosity inequalities apply to these closed jets, since
their gradients are nonzero.

The time derivatives of the cost satisfy
\begin{align*}
 &-\partial_t\bigl(\Lambda(s,t)h_\varepsilon(d)\bigr)
      -\frac{A}{\gamma}t^{r-1}
                       \bigl(\Lambda(s,t)h_\varepsilon'(d)\bigr)^{\gamma+1}\\
 &\quad=\frac{\Lambda(s,t)rt^{r-1}}{\gamma(t^r-s^r)}
       \left(h_\varepsilon(d)
        -\frac{\gamma}{\gamma+1}h_\varepsilon'(d)^{\gamma+1}\right).
\end{align*}
For the earlier endpoint, replace $-\partial_t$ by
$\partial_s$ and $t^{r-1}$ by $s^{r-1}$. Both remainders
are nonnegative and are $O(\varepsilon^2)$ uniformly at the
contacts by \eqref{gneq1:kernel-bounds} and the lower bounds
for $s$ and $t-s$.

For brevity, denote the two endpoint operators by
\[
 E_t=q^{\gamma-1}\left(AX_{nn}+B\sum_{i<n}X_{ii}\right)
                       +Bq^\gamma f_n(t),
\]
and
\[
 E_s=q^{\gamma-1}\left(AY_{nn}+B\sum_{i<n}Y_{ii}\right)
                       +Bq^\gamma f_n(s),
\]
where $f_n(\tau)=\langle\nabla f,e_n\rangle_{\alpha(\tau)}$,
$X_{ii}=X(e_i(t),e_i(t))$, and
$Y_{ii}=Y(e_i(s),e_i(s))$. The drift has a plus sign
because both gradients point along $-e_n$. Dividing the
viscosity inequalities by the positive diffusion coefficients
in \eqref{gneq1:weighted-equation} gives
\begin{align}
 E_t&\leq-\frac{\kappa}{t^r}
       -\frac{\varepsilon}
        {\theta t^{r-1}W(x_0,t)(b-t)^2}
       +O(\varepsilon^2),\label{gneq1:endpoint-jets}\\
 E_s&\geq-\frac{\kappa}{s^r}.\notag
\end{align}
We discarded the nonnegative remainder at the earlier
endpoint. The error at the later endpoint remains uniform
because $\theta W$ is bounded away from zero on the
contact region.

We next estimate the same operators by the second variation
of length. The longitudinal length in
\eqref{gneq1:variations} is $d+\ell(t^r-s^r)$ for small
$\ell$. The transverse first variations vanish, and the sum
of the transverse second variations is
\[
 \int_s^t\left((n-1)\frac{r^2\tau^{2r-2}}{|\alpha'|}
          -\tau^{2r}|\alpha'|\Ric(e_n,e_n)\right)\,d\tau.
\]
Contract the matrix inequality in \eqref{gneq1:jets}
against $(t^re_i(t),s^re_i(s))$, with weight $A$ for
$i=n$ and $B$ for $i<n$, and multiply by $q^{\gamma-1}$.
Choose $\rho>0$ so that
\[
 \rho q^{\gamma-1}(t^{2r}+s^{2r})
                  \bigl(A+(n-1)B\bigr)\leq\varepsilon^2.
\]
This choice is possible at every nonzero contact gradient,
including when $\gamma<1$. We obtain
\begin{align}
 &q^{\gamma-1}\left[
 t^{2r}\left(AX_{nn}+B\sum_{i<n}X_{ii}\right)
 -s^{2r}\left(AY_{nn}+B\sum_{i<n}Y_{ii}\right)\right]
       \notag\\
 &\quad\geq-A\Lambda^\gamma(h_\varepsilon')^{\gamma-1}
                       h_\varepsilon''(t^r-s^r)^2\notag\\
 &\quad{}-Bq^\gamma\int_s^t\left(
 (n-1)\frac{r^2\tau^{2r-2}}{|\alpha'|}
            -\tau^{2r}|\alpha'|\Ric(e_n,e_n)\right)\,d\tau
       -\varepsilon^2.\label{gneq1:second-variation}
\end{align}

The curvature and drift terms combine through
\[
 (\tau^{2r}f_n)'=2r\tau^{2r-1}f_n
                  +\tau^{2r}|\alpha'|\nabla^2f(e_n,e_n).
\]
For $N>n$, the assumption $\Ric_f^N\geq0$ and completion
of the square give
\begin{align*}
 \int_s^t\tau^{2r}|\alpha'|\Ric(e_n,e_n)\,d\tau
                         +[\tau^{2r}f_n]_s^t
 &\geq\int_s^t\left(
       \frac{\tau^{2r}|\alpha'|}{N-n}f_n^2
                   +2r\tau^{2r-1}f_n\right)\,d\tau\\
 &\quad\geq-(N-n)\int_s^t
                     \frac{r^2\tau^{2r-2}}{|\alpha'|}\,d\tau.
\end{align*}
When $N=n$, the same conclusion follows from $f$ being
constant and $\Ric\geq0$. Since
\[
 \int_s^t\frac{r^2\tau^{2r-2}}{|\alpha'|}\,d\tau
                         =\frac{(t^r-s^r)^2}{d},
\]
we may add the drift boundary term to
\eqref{gneq1:second-variation} and use
\eqref{gneq1:kernel-bounds} to conclude that
\begin{align}
 t^{2r}E_t-s^{2r}E_s
 &\geq-\Lambda^\gamma\left(\frac{A}{\gamma}+(N-1)B\right)
                       (t^r-s^r)^2-\varepsilon^2\notag\\
 &=-\kappa(t^r-s^r)-\varepsilon^2.
                       \label{gneq1:geometric-bound}
\end{align}
The last identity uses the definition of $\Lambda$ and the first
identity in \eqref{gneq1:cancellation}. If $B=0$, all transverse and drift
terms vanish, and this calculation uses no curvature
assumption.

Multiplying the two inequalities in
\eqref{gneq1:endpoint-jets} by $t^{2r}$ and $s^{2r}$,
respectively, and subtracting, we compare the resulting
upper bound with \eqref{gneq1:geometric-bound}. The common
term $-\kappa(t^r-s^r)$ cancels, leaving
\[
 0\leq-\frac{\varepsilon t^{r+1}}
                 {\theta W(x_0,t)(b-t)^2}+O(\varepsilon^2).
\]
The coefficient of $-\varepsilon$ has a uniform positive
lower bound: $t$ stays away from zero, $\theta W$ is
bounded above on the contact region, and $t<b$. Division
by $\varepsilon$ and passage to the limit give a
contradiction. Since the finite time $b<T$ was arbitrary,
\eqref{intro:general-pressure-harnack} follows. For
$0<\gamma<1$, apply the translated comparison at times
$s-\tau_0,t-\tau_0$ for $0<\tau_0<s$, and let
$\tau_0\downarrow0$ to recover the original solution.

It remains to obtain the differential inequality. Let a
$C^1$ function $\varphi$ touch $W$ from below at $(x,t)$.
With this later endpoint fixed, the comparison remains
valid when $W$ is replaced by $\varphi$ at nearby earlier
endpoints. Choose
\[
 y=\exp_x\left(\frac{A(\gamma+1)}{\gamma}t^{r-1}(t-s)
       |\nabla\varphi(x,t)|^{\gamma-1}\nabla\varphi(x,t)\right),
\]
where the vector is zero when the gradient vanishes.
Divide by $t-s$ and let $s\uparrow t$. The definition
of $\Lambda$ gives
\[
 \varphi_t-\frac{A}{\gamma}t^{r-1}
                   |\nabla\varphi|^{\gamma+1}\geq0.
\]
Undoing $W=t^{\theta\kappa}v$ and using
$r-1+\gamma\theta\kappa=0$ proves the first-order
viscosity inequality in the theorem. At a classical
noncritical point, \eqref{gneq1:pressure} identifies its
left-hand side with
$\theta v\left(\mathcal Q_\gamma[v]+\frac{\kappa}{t}\right)$.
The stated equivalence follows from $\theta v>0$.
\end{proof}

For $B=0$, $\kappa=\frac{1}{2\gamma}$, so the range condition
holds for every $\gamma>0$ with $\gamma\ne1$. The singular
regularity lemma also requires no curvature assumption in this case.

\begin{proof}[Proof of Theorem~\ref{thm:intro-weighted-p}]
The weighted $p$-Laplacian corresponds to
$\gamma=A=p-1$ and $B=1$. Consequently,
\[
 \kappa=\frac{N}{p+N(p-2)},\quad
 r=\frac{p}{p+N(p-2)}.
\]
The range condition in Theorem~\ref{intro:homogeneous-pressure}
is precisely $p>\frac{2N}{N+1}$. Substitution in that
theorem gives both \eqref{intro:weighted-p-harnack} and
\eqref{intro:weighted-p-differential}, together with the
classical formulation at noncritical points.
\end{proof}

The relation with the logarithmic case can also be seen
formally from the pressure comparison. As $\gamma\to1$,
\[
 t^{\theta\kappa}v
 =\frac{1}{\theta}+\log u+\kappa\log t+o(1),
\]
and
\[
 \kappa\longrightarrow\frac{1}{2}+\frac{(N-1)B}{2A}.
\]
The divergent constants cancel between the two endpoints, and
$\Lambda(s,t)\to\frac{1}{2A(t-s)}$. Thus the limiting cost and
time weight agree with Proposition~\ref{intro:prop-degree-one}.

\section{Negative curvature lower bounds}
\label{sec:negative-curvature}

A negative lower bound for the weighted Ricci tensor produces a
curvature term involving both the pressure and its gradient. To
control this term, we change the coefficient of the gradient term
and replace the power of time by a hyperbolic time weight. The
constants then depend on the initial amplitude. This gives the
following result in both the degenerate and singular ranges.

\begin{theorem}\label{intro:negative-pressure}
Let $(M^n,g,e^{-f}d\mu_g)$ be closed and connected, with
$\Ric_f^N\geq-Kg$ for $K\geq0$ and $n\leq N<\infty$,
and with $f$ constant when $N=n$. Let
$u\in C(M\times[0,T))$ be a positive viscosity solution of
$u_t=\Delta_{p,f}u$, for $p>\frac{2N}{N+1}$ and $p\ne2$.
Set
\[
 \theta=\frac{p-2}{p-1},\quad
 v=\frac{u^\theta}{\theta},\quad
 U=\max_{x\in M}u(x,0)^{p-2}.
\]
Choose $\delta<1$ with $\delta(p-2)>0$, sufficiently close to zero
that $\kappa_\delta$ below is positive, and define
\begin{align*}
 a&=1-\delta,\quad
 c_\delta=\frac{a\delta}{2}
           \left(p-2+\frac{p\delta}{N}\right),\\
 \kappa_\delta
 &=\left(p-2+\frac{pa}{N}
           -\frac{p^2a^2\delta^2}{N^2c_\delta}\right)^{-1},\\
 \omega
 &=\left(\frac{a^p(p-1)^{p-1}}
              {\kappa_\delta c_\delta^{p-1}}\right)^{\frac{1}{2}}
       K^{\frac{p}{2}}U.
\end{align*}
Let $S(t)=\frac{\sinh(\omega t)}{\omega}$ when $\omega>0$,
and $S(t)=t$ when $\omega=0$. Put
\[
 \mu(t)=S(t)^{\theta\kappa_\delta},\quad
 k_\delta(s,t)=\left(ap\int_s^t\mu(\tau)^{1-p}\,d\tau
        \right)^{-\frac{1}{p-1}}.
\]
Then
\begin{equation}\label{intro:negative-pressure-harnack}
 \mu(t)v(x,t)-\mu(s)v(y,s)
 +\frac{p-1}{p}k_\delta(s,t)d(x,y)^{\frac{p}{p-1}}\geq0
\end{equation}
for all $x,y\in M$ and $0<s<t<T$, and
\[
 v_t-a|\nabla v|^p
       +\theta\kappa_\delta\frac{S'(t)}{S(t)}v\geq0.
\]
At a classical noncritical point, this is equivalent to
\[
 \Delta_{p,f}v+\delta\frac{|\nabla v|^p}{\theta v}
       \geq-\kappa_\delta\frac{S'(t)}{S(t)}.
\]
\end{theorem}

The amplitude $U$ depends on the initial maximum for $p>2$ and on
the initial minimum for $p<2$. The sign of $\delta$ changes with
$p-2$, and the stated choice is possible throughout the range
$p>\frac{2N}{N+1}$. Earlier weighted differential Harnack
estimates under negative curvature are given
in~\cite{Wang2017,Wang2021}. In particular, the negative-curvature
estimates in~\cite{Wang2021} use an effective curvature quantity
involving the supremum of a pressure-gradient expression, with
positive pressure. Theorem~\ref{intro:negative-pressure} instead
gives the explicit comparison
\eqref{intro:negative-pressure-harnack} using the initial amplitude
and covers the negative pressure in the singular range as well.
No a priori gradient bound enters its constants.

In the proof, we control the additional second-variation term
by interpolating the positive endpoint values of $\theta\mu v$
and choosing the geodesic variation according to that interpolation.

The sign conditions on $\delta$ give $a>0$ and $c_\delta>0$.
Moreover, as $\delta\to0$ with $\delta(p-2)>0$,
\[
 \frac{1}{\kappa_\delta}\longrightarrow p-2+\frac{p}{N}>0.
\]
Thus $\kappa_\delta>0$ for sufficiently small $\delta$ of the
prescribed sign, throughout the stated range of $p$.

\begin{proof}[Proof of Theorem~\ref{intro:negative-pressure}]
We use $a,c_\delta,\kappa_\delta,U$ from the statement.
The amplitude bound in subsection~\ref{subsec:regularity} gives
\begin{equation}\label{neg:amplitude}
 u^{p-2}\leq U.
\end{equation}
For $p<2$, first translate the solution forward by a fixed
time $\tau_0>0$, retaining the notation $u,v,T$.
Lemma~\ref{lemma:singular-regularity} gives spatial and
temporal Lipschitz bounds for the pressure on compact
strips including the new initial time. The bound
\eqref{neg:amplitude} still holds with the original $U$,
which we keep throughout the proof.

Retain $\omega,S,\mu$, and $k_\delta$ from the theorem, and set
$z=\kappa_\delta\frac{S'}{S}$. Then
\begin{equation}\label{neg:z}
 \frac{\mu'}{\mu}=\theta z,\quad
 z'+\frac{z^2}{\kappa_\delta}=\kappa_\delta\omega^2,\quad
 z(t)\geq\frac{\kappa_\delta}{t}.
\end{equation}
The pressure equation is
$v_t=\theta v\Delta_{p,f}v+|\nabla v|^p$, so $W=\mu v$
satisfies
\begin{equation}\label{neg:weighted-equation}
 W_t=\theta\mu^{1-p}W\Delta_{p,f}W
            +\mu^{1-p}|\nabla W|^p+\theta zW,
 \quad\theta W>0.
\end{equation}
These identities hold in the viscosity sense by the same
change of variables used in Section~\ref{sec:homogeneous-power}.

Fix a finite time $b<T$ and suppose the comparison in
\eqref{intro:negative-pressure-harnack} fails for some $0<s<t<b$. Use the kernel
\eqref{gneq1:kernel} for $\gamma=p-1$, and define
\begin{equation}\label{neg:Z}
 Z_\varepsilon(x,t,y,s)
 =W(x,t)-W(y,s)+k_\delta(s,t)h_\varepsilon(d(x,y))
                         +\frac{\varepsilon}{b-t}.
\end{equation}
For small $\varepsilon$, this function has a value below
a fixed negative number. We localize its minimum as in
the preceding proof. For $p>2$, $W>0$ tends uniformly to
zero at time zero. Hence a fixed negative sublevel stays
away from $s=0$, and the divergence of $k_\delta(s,t)$
as $t-s\to0$, followed by uniform continuity of $W$,
keeps $t-s$ away from zero.

For $p<2$, $\mu$ is decreasing. By \eqref{neg:z},
\[
 \frac{\mu(s)}{\mu(t)}-1
 \geq-\theta\kappa_\delta\frac{t-s}{t},
\]
and
\[
 \mu(t)^{p-1}\int_s^t\mu(\tau)^{1-p}\,d\tau\leq t-s.
\]
Let $L$ be a common spatial and temporal Lipschitz constant
for the translated pressure on $M\times[0,b]$. Minimizing
the distance terms against $-Ld(x,y)$ now gives
\begin{align*}
 &\frac{1}{\mu(t)}\left(W(x,t)-W(y,s)
       +\frac{p-1}{p}k_\delta(s,t)
                         d(x,y)^{\frac{p}{p-1}}\right)\\
 &\quad\geq(t-s)\left(
       \frac{-\theta\kappa_\delta\min(-v)}{t}-L-aL^p\right),
\end{align*}
where the positive minimum of $-v$ is taken over this
strip. Choose $0<t_*<b$ so that the right-hand side is
nonnegative for $t\leq t_*$, and minimize
$Z_\varepsilon$ over $t\geq t_*$. On this region,
$-W(y,s)\to+\infty$ uniformly as $s\downarrow0$,
while $W(x,t)$ is bounded below. Thus $s$ stays away
from zero on a fixed negative sublevel. The distance cost
and uniform continuity then bound $t-s$ away from zero.
With these bounds, the regularized cost converges uniformly,
so the displayed inequality excludes the face $t=t_*$.
This treats the joint initial corner just as in
\eqref{gneq1:initial-corner}.

In both ranges, the terminal penalty excludes $t=b$ for
fixed $\varepsilon$. We therefore obtain a negative
interior minimum $(x_0,t,y_0,s)$, with uniform positive
lower bounds for $s$ and $t-s$. As in the preceding
proof, it cannot lie on the spatial diagonal: freezing the
earlier endpoint would give an admissible lower test with
zero spatial gradient and Hessian and time derivative
$-\frac{\varepsilon}{(b-t)^2}$, whereas
\eqref{neg:weighted-equation} requires that derivative to
be at least $\theta z(t)W(x_0,t)>0$.

At this contact, set
\[
 d=d(x_0,y_0)>0,\quad
 I=\int_s^t\mu(\tau)^{1-p}\,d\tau,\quad
 q=k_\delta(s,t)h_\varepsilon'(d)>0.
\]
With the contact pair fixed, $d,I,q$, and
$k_\delta=k_\delta(s,t)$ are constants, and kernel derivatives
are evaluated at $d$. The quantities $V,P$, and $\chi$ below are
functions on $[s,t]$. Interpolate the positive endpoint values
$\theta W$ by
\begin{align}
 V(\tau)&=\theta W(y_0,s)
  +\frac{\theta[W(x_0,t)-W(y_0,s)]}{I}
                      \int_s^\tau\mu(\xi)^{1-p}\,d\xi,
                      \label{neg:interpolation}\\
 P(\tau)&=\frac{\mu(\tau)^{1-p}q^p}{V(\tau)},\notag
\end{align}
for $s\leq\tau\leq t$.
In particular, $V>0$. The kernel bound
\eqref{gneq1:kernel-bounds} and
$k_\delta^{p-1}=\frac{1}{apI}$ give
$k_\delta h_\varepsilon\geq a(p-1)Iq^p$.
Since $Z_\varepsilon<0$, we have
$W(x_0,t)-W(y_0,s)<-k_\delta h_\varepsilon$.
Multiplication by $\delta\theta>0$ consequently gives
\begin{equation}\label{neg:interpolation-bounds}
 \delta P\frac{V'}{V}\leq-(p-2)\delta aP^2,
 \quad 0<\frac{V(\tau)}{\mu(\tau)}\leq U^{\frac{1}{p-1}}.
\end{equation}
For the second inequality, when $p>2$, $V$ decreases and
$\mu$ increases, so their ratio is bounded by its earlier
endpoint value. When $p<2$, $V$ increases and $\mu$
decreases, so the later endpoint gives the upper bound.
Both endpoint bounds follow from \eqref{neg:amplitude}.

Choose a positive function $\chi$ by
\begin{equation}\label{neg:index}
 \frac{\chi'}{\chi}=\frac{ap}{N}(z+\delta P),
 \quad\chi(s)=1.
\end{equation}
No monotonicity of $\chi$ is needed. This point matters
when $p<2$, since then $\delta<0$. On the contact
intervals, $V$ and $\mu$ are bounded above and away
from zero, while $q$ and $z$ are bounded. Thus $P$,
$\frac{\chi'}{\chi}$, $\chi$, and $\frac{1}{\chi}$
are uniformly bounded independently of small $\varepsilon$.

Choose a minimizing geodesic $\alpha:[s,t]\to M$ from
$y_0$ to $x_0$, parametrized by
\[
 |\alpha'(\tau)|=\frac{d\mu(\tau)^{1-p}}{I},
\]
and a parallel orthonormal frame with
$e_n=\frac{\alpha'}{|\alpha'|}$. Replace
\eqref{gneq1:variations} by
\begin{equation}\label{neg:variations}
 \alpha_i(\ell,\tau)
 =\exp_{\alpha(\tau)}\bigl(\ell\chi(\tau)e_i(\tau)\bigr),
\end{equation}
for $1\leq i\leq n$.
The construction in Section~\ref{sec:homogeneous-power}
again gives a single smooth upper support for distance:
divide the endpoint normal vectors by $\chi(s)$ and
$\chi(t)$, parallel transport and interpolate them
affinely, multiply by $\chi(\tau)$, and take the length
of the exponentiated curve. The support realizes all
variations in \eqref{neg:variations}, also at cut points.
We keep it fixed when varying the endpoint times.

Let $\psi$ be the last two terms of \eqref{neg:Z}, with
this support in place of distance. Lemma~\ref{lemma:2.1 two-point-jets}
gives matrices $X,Y$ for every $\rho>0$ such that
\begin{align}
 (-\psi_t,-D_x\psi,X)
       &\in\overline{\mathcal P}^{2,-}W(x_0,t),
       \label{neg:jets}\\
 (\psi_s,D_y\psi,Y)
       &\in\overline{\mathcal P}^{2,+}W(y_0,s),\notag\\
 \begin{pmatrix}-X&0\\0&Y\end{pmatrix}
       &\leq D^2_{x,y}\psi+\rho I_{2n}.\notag
\end{align}
The gradients are $-qe_n(t)$ and $-qe_n(s)$, so the
equation can be applied at both closed jets. Differentiating
the cost gives
\[
 -\partial_t(k_\delta h_\varepsilon) -a\mu(t)^{1-p}q^p
 =\frac{k_\delta\mu(t)^{1-p}}{(p-1)I} \left(h_\varepsilon
 -\frac{p-1}{p}(h_\varepsilon')^p\right) \geq0.
\]
For the earlier endpoint, replace $-\partial_t$ by
$\partial_s$ and $\mu(t)$ by $\mu(s)$. Both
remainders are uniformly $O(\varepsilon^2)$ by
\eqref{gneq1:kernel-bounds}.

With the endpoint-frame conventions of the preceding proof,
write
\[
 E_t=q^{p-2}\left((p-1)X_{nn}+\sum_{i<n}X_{ii}\right)
                                      +q^{p-1}f_n(t),
\]
and
\[
 E_s=q^{p-2}\left((p-1)Y_{nn}+\sum_{i<n}Y_{ii}\right)
                                      +q^{p-1}f_n(s).
\]
Subtracting the full gradient term from the viscosity
inequalities in \eqref{neg:weighted-equation} and dividing
by the positive diffusion coefficients yields
\begin{align}
 E_t&\leq-\mu(t)^{p-1}(z(t)+\delta P(t))
 -\frac{\varepsilon\mu(t)^{p-1}}{\theta W(x_0,t)(b-t)^2}
 +O(\varepsilon^2),\label{neg:endpoint-jets}\\
 E_s&\geq-\mu(s)^{p-1}(z(s)+\delta P(s)).\notag
\end{align}
We retain the negative terminal penalty at the later endpoint
and discard the nonnegative remainder at the earlier one.
The errors remain uniform since $\theta W$ is bounded away
from zero on the contact region.

We give the second variation calculation with the new
weight. The longitudinal length in \eqref{neg:variations}
is $d+\ell(\chi(t)-\chi(s))$ for small $\ell$.
This identity does not require $\chi(t)\geq\chi(s)$.
The transverse first variations vanish, and the sum of
their second variations is
\[
 \int_s^t\left((n-1)\frac{(\chi')^2}{|\alpha'|}
               -\chi^2|\alpha'|\Ric(e_n,e_n)\right)\,d\tau.
\]
Combining curvature and drift as before, the lower bound
$\Ric_f^N\geq-Kg$ gives
\begin{align}
 &\int_s^t\chi^2|\alpha'|\Ric(e_n,e_n)\,d\tau
                               +[\chi^2f_n]_s^t\notag\\
 &\quad\geq-K\int_s^t\chi^2|\alpha'|\,d\tau
              -(N-n)\int_s^t\frac{(\chi')^2}{|\alpha'|}\,d\tau.
                                      \label{neg:curvature}
\end{align}
Indeed, for $N>n$, differentiation of $\chi^2f_n$ and
the curvature bound leave the density terms
\[
 \frac{\chi^2|\alpha'|}{N-n}f_n^2+2\chi\chi'f_n
                  \geq-(N-n)\frac{(\chi')^2}{|\alpha'|}.
\]
For $N=n$, $f$ is constant and this term vanishes.

Contract \eqref{neg:jets} against
$(\chi(t)e_i(t),\chi(s)e_i(s))$, with weight $p-1$
in the longitudinal direction and $1$ in the transverse
directions, and multiply by $q^{p-2}$. The longitudinal
contribution is bounded using
\[
 (\chi(t)-\chi(s))^2
 \leq I\int_s^t\mu^{p-1}(\chi')^2\,d\tau,
\]
and
\[
 (p-1)(h_\varepsilon')^{p-2}h_\varepsilon''\leq1.
\]
It is therefore at most
$\frac{1}{ap}\int_s^t\mu^{p-1}(\chi')^2\,d\tau$
in magnitude. For the transverse contribution, use
$(h_\varepsilon')^{p-1}\leq d$,
$k_\delta^{p-1}=\frac{1}{apI}$, and
\eqref{neg:curvature}. Adding the drift boundary terms
gives
\[
 \chi(t)^2E_t-\chi(s)^2E_s \geq-\frac{N}{ap}\int_s^t\mu^{p-1}(\chi')^2\,d\tau
 -Kq^{p-1}\frac{d}{I} \int_s^t\chi^2\mu^{1-p}\,d\tau -\varepsilon^2.
\]
Here we choose the matrix parameter so that
\[
 \rho q^{p-2}\bigl(\chi(t)^2+\chi(s)^2\bigr)(n+p-2)
                                      \leq\varepsilon^2.
\]
The nonzero contact gradient allows this choice also for
$p<2$. Combining the geometric bound with
\eqref{neg:endpoint-jets}, we obtain
\begin{align}
 0\leq{}&-[\chi^2\mu^{p-1}(z+\delta P)]_s^t
       +\frac{N}{ap}\int_s^t\mu^{p-1}(\chi')^2\,d\tau
                                      \notag\\
 &+Kq^{p-1}\frac{d}{I}
                 \int_s^t\chi^2\mu^{1-p}\,d\tau\notag\\
 &-\frac{\varepsilon\chi(t)^2\mu(t)^{p-1}}
              {\theta W(x_0,t)(b-t)^2}
       +O(\varepsilon^2).\label{neg:combined}
\end{align}

The choice of $\chi$ and the interpolation estimate now
control the curvature term. By \eqref{neg:interpolation},
\[
 \frac{P'}{P}=-(p-2)z-\frac{V'}{V}.
\]
Using \eqref{neg:index} and
\eqref{neg:interpolation-bounds}, we calculate
\begin{align*}
 & \frac{[\chi^2\mu^{p-1}(z+\delta P)]'
        -\frac{N}{ap}\mu^{p-1}(\chi')^2}
        {\chi^2\mu^{p-1}} \\
 = &\quad  z'+(p-2)z^2+\frac{ap}{N}(z+\delta P)^2
                                  -\delta P\frac{V'}{V}\\
 \geq &\quad z'+\left(p-2+\frac{ap}{N}\right)z^2
                  +\frac{2ap\delta}{N}zP+2c_\delta P^2\\
 = &\quad z'+\frac{z^2}{\kappa_\delta}+c_\delta P^2
       +c_\delta\left(P+\frac{ap\delta}{Nc_\delta}z\right)^2\\
 \geq &\quad \kappa_\delta\omega^2+c_\delta P^2.
\end{align*}
In the last two lines we used the definitions of
$c_\delta,\kappa_\delta$ and \eqref{neg:z}. All
functions differentiated in this calculation are auxiliary
functions; no derivatives of the solution are taken.

The remaining kernel error satisfies
\[
 0\leq q^{p-1}\frac{d}{I}-apq^{2p-2}
 =\frac{q^{p-1}}{I}
                 \bigl(d-(h_\varepsilon')^{p-1}\bigr)
 =O(\varepsilon^2).
\]
Also, \eqref{neg:interpolation-bounds} gives
\[
 q^{2p-2}\mu^{1-p} =\mu^{p-1}P^{\frac{2(p-1)}{p}}
 \left(\frac{V}{\mu}\right)^{\frac{2(p-1)}{p}} \leq
 U^{\frac{2}{p}}\mu^{p-1}P^{\frac{2(p-1)}{p}}.
\]
Consequently, the curvature contribution in
\eqref{neg:combined} is at most
\[
 apKU^{\frac{2}{p}}\int_s^t
       \chi^2\mu^{p-1}P^{\frac{2(p-1)}{p}}\,d\tau
       +O(\varepsilon^2).
\]
For $P\geq0$, we have
\begin{equation}\label{neg:young}
 apKU^{\frac{2}{p}}P^{\frac{2(p-1)}{p}}-c_\delta P^2
 \leq\frac{a^p(p-1)^{p-1}K^pU^2}{c_\delta^{p-1}}.
\end{equation}
When $K>0$, this follows by maximizing the left-hand
side: its maximum occurs at
$P^{\frac{2}{p}}=\frac{a(p-1)KU^{\frac{2}{p}}}{c_\delta}$.
For $K=0$, the inequality is immediate.
By the definition of $\omega$, the right-hand side of
\eqref{neg:young} is $\kappa_\delta\omega^2$, also when
$K=0$. Substituting these bounds into \eqref{neg:combined}
therefore gives
\[
 0\leq-\frac{\varepsilon\chi(t)^2\mu(t)^{p-1}}
                  {\theta W(x_0,t)(b-t)^2}+O(\varepsilon^2).
\]
The coefficient of $-\varepsilon$ has a uniform positive
lower bound: $\chi$ and $\mu$ are bounded away from zero,
$\theta W$ is positive and bounded above, and $t<b$.
Dividing by $\varepsilon$ and letting $\varepsilon\downarrow0$
gives a contradiction.

Since $b<T$ was arbitrary, the comparison holds for all
$0<s<t<T$. For $p<2$, remove the time translation by applying
it at $s-\tau_0,t-\tau_0$ and letting $\tau_0\downarrow0$.
The constant $U$ has retained its original value throughout.

For the differential inequality, let a $C^1$ function
$\varphi$ touch $W$ from below at $(x,t)$, and choose
\[
 y=\exp_x\bigl(ap\mu(t)^{1-p}(t-s)
       |\nabla\varphi(x,t)|^{p-2}\nabla\varphi(x,t)\bigr),
\]
with the vector defined to be zero at a zero gradient.
Replace $W$ by $\varphi$ at the nearby earlier endpoint
of the comparison, divide by $t-s$, and let $s\uparrow t$.
Since
$k_\delta(s,t)^{p-1}(t-s)\to\frac{\mu(t)^{p-1}}{ap}$,
we obtain
\[
 \varphi_t-a\mu(t)^{1-p}|\nabla\varphi|^p\geq0.
\]
Undoing $W=\mu v$ and using
$\frac{\mu'}{\mu}=\theta z$ gives the asserted
first-order viscosity inequality. At a classical
noncritical point, the pressure equation identifies its
left-hand side with
\[
 \theta v\left(\Delta_{p,f}v
           +\delta\frac{|\nabla v|^p}{\theta v}+z\right).
\]
Since $\theta v>0$, this also proves the stated
classical formulation.
\end{proof}

When $K=0$, $S(t)=t$. Letting $\delta\to0$ with
$\delta(p-2)>0$ gives $\kappa_\delta\to\kappa$ and
$\mu(t)\to t^{\theta\kappa}$, while
\[
 k_\delta(s,t)\longrightarrow
 \left(\frac{r}{p(t^r-s^r)}\right)^{\frac{1}{p-1}},
 \quad r=1-(p-2)\kappa=\frac{p}{p+N(p-2)}.
\]
Thus the two conclusions reduce to those of
Theorem~\ref{thm:intro-weighted-p}.

The logarithmic argument for $p=2$ gives a separate weighted heat
estimate without amplitude dependence.

\begin{corollary}\label{intro:negative-weighted-heat}
Under the geometric assumptions of
Theorem~\ref{intro:negative-pressure}, let
$u\in C(M\times[0,T))$ be a positive viscosity solution of
$u_t=\Delta_f u$. For $0<\delta<1$, set
\[
 a=1-\delta,\quad
 \omega=\frac{aK}{\delta},
\]
and define $S(t)=\frac{\sinh(\omega t)}{\omega}$ for
$\omega>0$, with $S(t)=t$ when $\omega=0$. Then
\[
 u(x,t)\geq u(y,s)
 \left(\frac{S(s)}{S(t)}\right)^{\frac{N}{2a}}
 \exp\!\left(-\frac{d(x,y)^2}{4a(t-s)}\right)
\]
for all $x,y\in M$ and $0<s<t<T$, and $v=\log u$ satisfies
\[
 v_t-a|\nabla v|^2
       +\frac{N}{2a}\frac{S'(t)}{S(t)}\geq0.
\]
\end{corollary}

The corollary gives the weighted heat estimate obtained by the
same comparison method. We do not claim that its curvature term
is optimal; refined linear estimates with other coefficient choices
are given in~\cite{LiXu2011,BakryBolleyGentil2017,ZhangSharp2025}.

\begin{proof}[Proof of Corollary~\ref{intro:negative-weighted-heat}]
Retain $a,\delta,\omega$, and $S$ from the corollary, and put
\[
 z=\frac{N}{2a}\frac{S'}{S},\quad
 W=\log u+\frac{N}{2a}\log S,\quad
 \chi(\tau)=e^{\omega\tau}S(\tau).
\]
Then $W_t=\Delta_fW+|\nabla W|^2+z$. The time weight
satisfies
\[
 \chi'>0,\quad \chi'\geq\omega\chi,\quad
 (\chi^2z)'=\frac{N}{2a}(\chi')^2.
\]
Indeed, for $\omega>0$,
$\chi=\frac{e^{2\omega\tau}-1}{2\omega}$ and
$\chi^2z=\frac{N}{8a\omega}(e^{4\omega\tau}-1)$.
For $\omega=0$, these identities follow from
$\chi(\tau)=\tau$ and $z(\tau)=\frac{N}{2a\tau}$.

We prove
\begin{equation}\label{neg:heat-comparison}
 W(x,t)-W(y,s)+\frac{d(x,y)^2}{4a(t-s)}\geq0.
\end{equation}
Fix a finite time $b<T$ and suppose the comparison fails
for some $0<s<t<b$. For sufficiently small
$\varepsilon>0$, the function
\[
 Z_\varepsilon(x,t,y,s)
 =W(x,t)-W(y,s)+\frac{d(x,y)^2}{4a(t-s)}
                         +\frac{\varepsilon}{b-t}
\]
has a negative value. It attains a negative minimum at
interior times. To see this, $\log u$ is bounded and
uniformly continuous on $M\times[0,b]$, and
$\frac{S(t)}{S(s)}\geq\frac{t}{s}$. If $t-s\to0$
on a fixed negative sublevel, the nonnegative time term
and distance cost force $d(x,y)\to0$, contradicting
uniform continuity. If $s\to0$ with $t-s$ bounded below,
the time term tends to $+\infty$. Finally, the terminal
penalty excludes $t=b$.

Fix such a minimum and keep the notation $(x,t,y,s)$
for it. Let $\alpha:[s,t]\to M$ be a constant-speed
minimizing geodesic from $y$ to $x$. When $x=y$, take
the constant geodesic. To construct a smooth upper support
for squared distance, interpolate small endpoint vectors
$U\in T_yM$ and $V\in T_xM$ along $\alpha$ with the field
\begin{equation}\label{eq:s8-endpoint-field}
 J_{U,V}(r)=\chi(r)\left(
 \frac{t-r}{(t-s)\chi(s)}P_{s,r}U
 +\frac{r-s}{(t-s)\chi(t)}P_{t,r}V\right),
\end{equation}
for $s\leq r\leq t$, where $P$ is parallel transport.
Set $\Gamma_{U,V}(r)=\exp_{\alpha(r)}J_{U,V}(r)$. The smooth function
\[
 \mathscr E(U,V)=(t-s)\int_s^t|\partial_r\Gamma_{U,V}|^2\,dr
\]
majorizes the squared distance between the new endpoints
by Cauchy-Schwarz and agrees with it at $U=V=0$.
This remains true at cut points and on the spatial diagonal.
The construction applies to any smooth positive weight on $[s,t]$.

Replace $d^2$ by this support in the penalty and apply
Lemma~\ref{lemma:2.1 two-point-jets}. Keep the support's
construction parameters fixed when differentiating the
endpoint times. The first variation gives the common
endpoint gradient $-\frac{\dot\alpha}{2a}$ in parallel
frames. Set $q=\frac{|\dot\alpha|}{2a}$, allowing
$q=0$. For every $\rho>0$, the resulting spatial jets
$X,Y$ satisfy
\[
 \begin{pmatrix}X&0\\0&-Y\end{pmatrix}
 \geq-\nabla^2_{x,y}\psi-\rho I_{2n},
\]
where $\psi$ is the supported penalty. The endpoint
viscosity inequalities are
\[
 E_t:=\tr X+\frac{1}{2a}\langle\nabla f,\dot\alpha\rangle_t
 \leq-z(t)-\delta q^2-\frac{\varepsilon}{(b-t)^2},
\]
and
\[
 E_s:=\tr Y+\frac{1}{2a}\langle\nabla f,\dot\alpha\rangle_s
 \geq-z(s)-\delta q^2.
\]
Here the time derivatives of the distance cost contribute
$aq^2$ at both endpoints. The operator is continuous also
at a zero gradient, so these inequalities include $x=y$.

Contract the matrix inequality against
$(\chi(t)e_i(t),\chi(s)e_i(s))$ and sum over a parallel
orthonormal frame. The second variation of energy gives
\begin{align*}
 \chi(t)^2\operatorname{tr}X-\chi(s)^2\operatorname{tr}Y
 &\geq-\frac{n}{2a}\int_s^t(\chi')^2\,dr
       +\frac{1}{2a}\int_s^t\chi^2
                          \Ric(\dot\alpha,\dot\alpha)\,dr\\
 &\quad{}-\rho n\bigl(\chi(t)^2+\chi(s)^2\bigr).
\end{align*}
The endpoint variation curves are geodesics, so no endpoint
acceleration terms occur. For $N>n$, differentiating
$\chi^2\langle\nabla f,\dot\alpha\rangle$ and using
$\Ric_f^N\geq-Kg$ yields
\begin{align*}
 &\int_s^t\chi^2\Ric(\dot\alpha,\dot\alpha)\,dr
           +[\chi^2\langle\nabla f,\dot\alpha\rangle]_s^t\\
 &\quad\geq-K|\dot\alpha|^2\int_s^t\chi^2\,dr
       +\int_s^t\left(
       \frac{\chi^2}{N-n}\langle\nabla f,\dot\alpha\rangle^2
       +2\chi\chi'\langle\nabla f,\dot\alpha\rangle\right)\,dr\\
 &\quad\geq-K|\dot\alpha|^2\int_s^t\chi^2\,dr
                       -(N-n)\int_s^t(\chi')^2\,dr.
\end{align*}
For $N=n$, the same conclusion follows because $f$ is
constant. These bounds also hold for the constant geodesic.
Adding the drift terms and using
$\frac{|\dot\alpha|^2}{2a}=2aq^2$, we obtain
\[
 \chi(t)^2E_t-\chi(s)^2E_s
 \geq-\frac{N}{2a}\int_s^t(\chi')^2\,dr
 -2aKq^2\int_s^t\chi^2\,dr
 -\rho n\bigl(\chi(t)^2+\chi(s)^2\bigr).
\]
Combining this with the endpoint viscosity inequalities
and using $(\chi^2z)'=\frac{N}{2a}(\chi')^2$ gives
\begin{align*}
 \frac{\varepsilon\chi(t)^2}{(b-t)^2}
 &\leq q^2\left(2aK\int_s^t\chi^2\,dr
                         -\delta[\chi^2]_s^t\right)
       +\rho n\bigl(\chi(t)^2+\chi(s)^2\bigr)\\
 &\leq\rho n\bigl(\chi(t)^2+\chi(s)^2\bigr).
\end{align*}
For the last inequality we used
\[
 \delta[\chi^2]_s^t
 =2\delta\int_s^t\chi\chi'\,dr
 \geq2\delta\omega\int_s^t\chi^2\,dr
 =2aK\int_s^t\chi^2\,dr.
\]
The contact and $\varepsilon$ are fixed. Letting
$\rho\downarrow0$ is therefore a contradiction.
The argument includes $K=0$ without a limiting step.
Since $b<T$ was arbitrary, \eqref{neg:heat-comparison}
holds for all $0<s<t<T$. Substituting $W$ and
exponentiating proves the Harnack inequality.

Finally, let a $C^1$ function $\varphi$ touch $W$ from
below at $(x,t)$ and take
$y=\exp_x(2a(t-s)\nabla\varphi(x,t))$ in
\eqref{neg:heat-comparison}. Replacing $W$ by
$\varphi$ at the earlier endpoint, dividing by $t-s$,
and letting $s\uparrow t$ give
$\varphi_t-a|\nabla\varphi|^2\geq0$.
Undoing $W=v+\frac{N}{2a}\log S$ proves the
first-order viscosity inequality.
\end{proof}

\section{Fully nonlinear equations on curved manifolds}
\label{sec:harnack-manifolds}

We extend the flat-space estimates to closed Riemannian manifolds.
We combine the matrix comparison from
Section~\ref{sec:fully} with the energy upper support used in the
weighted heat proof. The second variation now contains a curvature
term with matrix weights. This is
the reason for the sectional curvature assumption in these theorems.

Throughout this section, $M$ is closed and connected, and $F$ is
continuous, positively homogeneous of degree one, and uniformly
elliptic with constants $0<\lambda\leq\Lambda$. Here $D^2u$
denotes the Riemannian Hessian. We also assume
\begin{equation}\label{eq:s8-invariance}
 F(O^TXO)=F(X),
\end{equation}
for $O\in O(n)$ and $X\in\Sym(n)$.
Thus the equation $u_t=F(\nabla^2u)$ is independent of the
orthonormal frame; see~\cite{HarveyLawson2011}. In particular, we
may identify the endpoint tangent spaces by parallel transport and
use the same matrix function at both endpoints. No convexity or
concavity of $F$ is required.

Kim and Lee~\cite{KimLee2014} proved Harnack inequalities for a
general uniformly parabolic viscosity class under sectional curvature
bounds, with scale-independent constants when the sectional curvature
is nonnegative. Here positive homogeneity gives explicit global time
and distance factors, together with first-order differential
inequalities.

\begin{theorem}\label{intro:thm-manifold-nonnegative}
Let $(M^n,g)$ be closed and connected with nonnegative sectional
curvature. Let $F\colon\Sym(n)\to\R$ be continuous and satisfy
\eqref{intro:fn-homogeneity}, \eqref{intro:fn-ellipticity}, and
\eqref{eq:s8-invariance}, with constants $0<\lambda\leq\Lambda$.
Set $C=\frac{n\Lambda}{2\lambda}$. Every positive
continuous viscosity solution of \eqref{eq 1.1 fully nonlinear}
on $M\times(0,T)$ satisfies
\[
 u(x,t)\geq u(y,s)\left(\frac{s}{t}\right)^C
 \exp\!\left(-\frac{d(x,y)^2}{4\lambda(t-s)}\right),
\]
for $x,y\in M$ and $0<s<t<T$.
Its logarithm $v=\log u$ satisfies
\[
 v_t-\lambda|\nabla v|^2+\frac{C}{t}\geq0
\]
in the first-order viscosity sense.
\end{theorem}

\begin{theorem}\label{intro:thm-manifold-negative}
Let $(M^n,g)$ be closed and connected with $\operatorname{Sec}\geq-K$,
where $K>0$. Let $F\colon\Sym(n)\to\R$ be continuous and satisfy
\eqref{intro:fn-homogeneity}, \eqref{intro:fn-ellipticity}, and
\eqref{eq:s8-invariance}, with constants $0<\lambda\leq\Lambda$.
For any $\alpha>1$, set
\[
 C=\frac{n\Lambda}{2\lambda},\quad
 c=\frac{(n-1)\Lambda K}{\alpha-1}.
\]
Every positive continuous viscosity solution of
\eqref{eq 1.1 fully nonlinear} on $M\times(0,T)$ satisfies
\[
 u(x,t)\geq u(y,s)
 \left(\frac{\sinh\!\left(\frac{cs}{2}\right)}
 {\sinh\!\left(\frac{ct}{2}\right)}\right)^{\alpha C}
 \exp\!\left(-\frac{\alpha d(x,y)^2}{4\lambda(t-s)}\right),
\]
and
\[
 v_t-\frac{\lambda}{\alpha}|\nabla v|^2
 +\frac{\alpha Cc}{2}\coth\!\left(\frac{ct}{2}\right)\geq0,
 \quad v=\log u,
\]
for all $x,y\in M$ and $0<s<t<T$, with the differential inequality
understood in the first-order viscosity sense. In particular,
\[
 u(x,t)\geq u(y,s)\left(\frac{s}{t}\right)^{\alpha C}
 \exp\!\left(-\frac{\alpha d(x,y)^2}{4\lambda(t-s)}
             -\frac{\alpha Cc}{2}(t-s)\right),
\]
and
\[
 v_t-\frac{\lambda}{\alpha}|\nabla v|^2
 +\frac{\alpha C}{t}+\frac{\alpha Cc}{2}\geq0.
\]
When $c=0$, the hyperbolic sine ratio and
$\frac{c}{2}\coth\!\left(\frac{ct}{2}\right)$ are interpreted as
$\frac{s}{t}$ and $\frac{1}{t}$, respectively.
\end{theorem}

\begin{proof}[Proof of Theorems~\ref{intro:thm-manifold-nonnegative}
and~\ref{intro:thm-manifold-negative}]
We first assume $\operatorname{Sec}\geq-K$ for $K>0$ and prove
the estimate for a fixed $\alpha>1$. The nonnegative-curvature
result will follow by taking limits. Retain $C$ and $c$ from
Theorem~\ref{intro:thm-manifold-negative}, and set $v=\log u$.
The test-function calculation in Section~\ref{sec:fully} gives
\[
 v_t=G(\nabla v,\nabla^2v),\quad G(p,X)=F(X+p\otimes p).
\]
Indeed, $\nabla^2(e^\varphi)=e^\varphi
(\nabla^2\varphi+d\varphi\otimes d\varphi)$, and positive
homogeneity permits the common factor to be removed at upper and
lower contacts.

Fix $0<a<b<T$. For $r>a$, define the time weights
\[
 \chi(r)=\frac{e^{c(r-a)}-1}{c},
\]
and
\[
 L(r)=\frac{\alpha Cc}{2}\coth\!\left(\frac{c(r-a)}{2}\right).
\]
When $c=0$, set $\chi(r)=r-a$ and $L(r)=\frac{\alpha C}{r-a}$.
For $a<s<t<b$, we claim that
\begin{equation}\label{eq:s8-comparison}
 Z(x,y,t,s):=v(x,t)-v(y,s)+\int_s^t L(r)\,dr
 +\frac{\alpha d(x,y)^2}{4\lambda(t-s)}\geq0.
\end{equation}
Suppose the claim fails. For small
$\varepsilon>0$, the function
\[
 Z_\varepsilon=Z+\frac{\varepsilon}{b-t}
\]
has a negative minimum at interior times. Here is the required
boundary localization. The function $v$ is bounded and uniformly
continuous on $M\times[a,b]$, and all the penalty terms are
nonnegative. The last term excludes $t\uparrow b$. If
$s\downarrow a$ while $t-s$ stays positive, the time integral
tends to infinity. If $t-s\to0$ along a sequence where
$Z_\varepsilon<0$, boundedness first forces $d(x,y)\to0$;
uniform continuity then gives $v(x,t)-v(y,s)\to0$. Such a
sequence cannot approach a negative infimum.

Fix a minimizing point and retain the notation $(x,y,t,s)$ for
it. For $r>a$, the time weights satisfy
\begin{equation}\label{eq:s8-weight-identities}
 \chi'>0,\quad \chi'\geq c\chi,\quad
 (\chi^2L)'=\alpha C(\chi')^2.
\end{equation}
The last identity matches the time penalty to the second variation.

\smallskip
\noindent\emph{The geodesic variation.}
Let $\gamma_0\colon[s,t]\to M$ be a constant-speed
minimizing geodesic from $y$ to $x$. Use the energy upper support $\mathscr E$ from the proof of
Corollary~\ref{intro:negative-weighted-heat}, with the present
weight $\chi$ and geodesic $\gamma_0$. Its endpoint field is
\eqref{eq:s8-endpoint-field}. This construction is valid at cut
points and, using the constant geodesic, when $x=y$. Thus the
spatial diagonal requires no separate regularization.

Replace $d^2$ in the penalty by this upper support and denote the
resulting smooth penalty, including $\frac{\varepsilon}{b-t}$,
by $\psi$. The local minimum is preserved. Keep the support and
its construction parameters fixed when differentiating the endpoint
times. In parallel frames, the first variation gives
\[
 -D_x\psi=D_y\psi=p,\quad
 p=-\frac{\alpha}{2\lambda}\dot\gamma_0,\quad
 H=\frac{\lambda}{\alpha}|p|^2.
\]
The time derivatives are
\[
 -\psi_t=-L(t)+H-\frac{\varepsilon}{(b-t)^2},
\]
and
\[
 \psi_s=-L(s)+H.
\]
The vector $p$ is understood at the respective endpoints by
parallel transport.

Let $A$ be any symmetric matrix with
$\lambda I\leq A\leq\Lambda I$. Choose a parallel orthonormal
frame $e_i$ diagonalizing $A$, and write its eigenvalues as $a_i$.
For the paired vector
\[
 E_i=(\chi(t)e_i(t),\chi(s)e_i(s)),
\]
the field in \eqref{eq:s8-endpoint-field} is $\chi e_i$.
The endpoint curves are geodesics, so their accelerations make
no contribution to the second variation. With the convention
$\langle R(e,E)E,e\rangle=\operatorname{Sec}(e\wedge E)$ for
orthonormal $e,E$, we obtain
\begin{align}\label{eq:s8-second-variation}
 \sum_i a_i\nabla^2_{x,y}\psi(E_i,E_i)
 &=\frac{\alpha}{2\lambda}
 \int_s^t\left\{
 \tr(A)(\chi')^2
 -\chi^2\sum_i a_i
 \langle R(e_i,\dot\gamma_0)\dot\gamma_0,e_i\rangle
 \right\}\,dr\notag\\
 &\leq\alpha C\int_s^t(\chi')^2\,dr
 +2(n-1)\Lambda KH\int_s^t\chi^2\,dr.
\end{align}
To verify the curvature bound when $\dot\gamma_0\ne0$, put
$E=\frac{\dot\gamma_0}{|\dot\gamma_0|}$. Since $a_i>0$,
the sectional curvature assumption yields
\[
 \sum_i a_i\langle R(e_i,\dot\gamma_0)\dot\gamma_0,e_i\rangle
 \geq-K|\dot\gamma_0|^2 \bigl(\tr A-\langle AE,E\rangle\bigr)
 \geq-(n-1)\Lambda K|\dot\gamma_0|^2.
\]
Now use $\tr A\leq n\Lambda$ and
$\frac{\alpha}{2\lambda}|\dot\gamma_0|^2=2H$.
If $\dot\gamma_0=0$, the curvature term is zero.

\smallskip
\noindent\emph{The matrix comparison.}
Apply Lemma~\ref{lemma:2.1 two-point-jets} in normal coordinates
at the endpoints. The corresponding Riemannian elliptic maximum principle
for semicontinuous functions is given in~\cite[Theorem~2.10]{AzagraFerreraSanz2008};
the lemma explains the passage to independent parabolic times.
At the fixed contact, choose its matrix parameter sufficiently
small that the spatial error is at most $\rho I_{2n}$. We obtain
closed jets with spatial Hessians $X,Y$ satisfying
\begin{align}\label{eq:s8-jets}
 G(p,X)&\leq-L(t)+H-\frac{\varepsilon}{(b-t)^2},\notag\\
 G(p,Y)&\geq-L(s)+H,\notag\\
 \begin{pmatrix}X&0\\0&-Y\end{pmatrix}
 &\geq-\nabla^2_{x,y}\psi-\rho I_{2n}.
\end{align}
The later endpoint supplies a closed subjet, and the earlier one a
closed superjet. The viscosity inequalities hold on these jets
because $G$ is continuous. Assumption~\eqref{eq:s8-invariance}
allows both endpoint inequalities to be written in the chosen
parallel frames.

Multiply the first two inequalities in \eqref{eq:s8-jets} by
$\chi(t)^2$ and $\chi(s)^2$, respectively, and subtract.
Using \eqref{eq:s8-weight-identities}, we find
\begin{align}\label{eq:s8-time-bound}
 &\chi(t)^2G(p,X)-\chi(s)^2G(p,Y)\notag\\
 &\quad\leq-\alpha C\int_s^t(\chi')^2\,dr
 +\bigl(\chi(t)^2-\chi(s)^2\bigr)H
 -\frac{\varepsilon\chi(t)^2}{(b-t)^2}.
\end{align}
Choose a minimizer $A$ in the lower Pucci operator applied to
\[
 \chi(t)^2(X+p\otimes p)
       -\chi(s)^2(Y+p\otimes p).
\]
Positive homogeneity and uniform ellipticity give
\begin{align}\label{eq:s8-ellipticity-bound}
 &\chi(t)^2G(p,X)-\chi(s)^2G(p,Y)\notag\\
 &\quad\geq\tr\!\left(A
       \bigl(\chi(t)^2X-\chi(s)^2Y\bigr)\right)
 +\bigl(\chi(t)^2-\chi(s)^2\bigr)\langle Ap,p\rangle.
\end{align}
Contract the last inequality in \eqref{eq:s8-jets} against
$E_i$ and sum with the positive weights $a_i$. Using
\eqref{eq:s8-second-variation} and $\tr A\leq n\Lambda$, we obtain
\begin{align*}
 &\tr\!\left(A\bigl(\chi(t)^2X-\chi(s)^2Y\bigr)\right)\\
 &\quad\geq-\alpha C\int_s^t(\chi')^2\,dr
 -2(n-1)\Lambda KH\int_s^t\chi^2\,dr
 -\rho n\Lambda\bigl(\chi(t)^2+\chi(s)^2\bigr).
\end{align*}
The choice of $c$ absorbs the curvature term. Indeed,
\begin{equation}\label{eq:s8-curvature-absorption}
 (\alpha-1)\bigl(\chi(t)^2-\chi(s)^2\bigr) =2(\alpha-1)\int_s^t\chi\chi'\,dr
 \geq2(n-1)\Lambda K\int_s^t\chi^2\,dr,
\end{equation}
since $\chi'\geq c\chi$. Also,
$\langle Ap,p\rangle\geq\lambda|p|^2=\alpha H$, and
$\chi(t)>\chi(s)$. Combining
\eqref{eq:s8-time-bound}, \eqref{eq:s8-ellipticity-bound}, and
\eqref{eq:s8-curvature-absorption} therefore gives
\[
 \frac{\varepsilon\chi(t)^2}{(b-t)^2}
 \leq\rho n\Lambda\bigl(\chi(t)^2+\chi(s)^2\bigr).
\]
Letting $\rho\downarrow0$ yields a contradiction. The matrices
$A,X,Y$ may depend on $\rho$, but the error is uniform because
$A\leq\Lambda I$. Thus no limit of the individual jets is needed.
This proves \eqref{eq:s8-comparison}. For fixed $0<s<t<T$,
choose $t<b<T$ and let $a\downarrow0$. Since
\[
 \int_s^t L(r)\,dr\longrightarrow
 \alpha C\log\frac{\sinh\!\left(\frac{ct}{2}\right)}
                         {\sinh\!\left(\frac{cs}{2}\right)},
\]
with the usual interpretation for $c=0$, exponentiating gives
the Harnack inequality in Theorem~\ref{intro:thm-manifold-negative}.

\smallskip
\noindent\emph{The differential inequality and the limits.}
Let $\varphi$ touch $v$ from below at $(x,t)$. In the logarithmic
Harnack inequality, take $s=t-h$ and $y=\exp_x(h\xi)$, where
$\xi\in T_xM$. The lower contact allows $v(y,s)$ to be replaced
by $\varphi(y,s)$. Divide by $h$ and let $h\downarrow0$ to get
\[
 \varphi_t-\langle\nabla\varphi,\xi\rangle
 +\frac{\alpha}{4\lambda}|\xi|^2
 +\frac{\alpha Cc}{2}\coth\!\left(\frac{ct}{2}\right)\geq0.
\]
Choosing $\xi=\frac{2\lambda}{\alpha}\nabla\varphi$ proves
the first-order inequality. For $c>0$ and $r>0$,
\[
 \frac{c}{2}\coth\!\left(\frac{cr}{2}\right)
 =\frac{c}{2}+\frac{c}{e^{cr}-1}
 \leq\frac{c}{2}+\frac{1}{r}.
\]
The bound extends to $c=0$ by continuity and gives the elementary
differential estimate. Integrating the
same inequality from $s$ to $t$ gives the elementary Harnack
estimate.

If $\operatorname{Sec}\geq0$, the estimates hold for every
$K>0$. First let $K\downarrow0$ with $\alpha>1$ fixed, using
\[
 \frac{\sinh\!\left(\frac{cs}{2}\right)}
      {\sinh\!\left(\frac{ct}{2}\right)}
 \longrightarrow\frac{s}{t},
\]
and
\[
 \frac{c}{2}\coth\!\left(\frac{ct}{2}\right)
 \longrightarrow\frac{1}{t}.
\]
Then let $\alpha\downarrow1$. This proves
Theorem~\ref{intro:thm-manifold-nonnegative}, including its
first-order consequence. Before sending $a$ to zero, the endpoint
weights tend to $t-a$ and $s-a$, as in Section~\ref{sec:fully}.
\end{proof}

\begin{remark}\label{rem:s8-curvature}
Ricci curvature controls the unweighted trace of the curvature
operator, whereas the contraction in
\eqref{eq:s8-second-variation} has the potentially unequal weights
$a_i$. Consequently, a Ricci lower bound alone need not provide
the required control when $\lambda<\Lambda$.

For $F(X)=\tr X$, take $\lambda=\Lambda=1$. Then $A=I$, and
the curvature contraction is exactly
$\Ric(\dot\gamma_0,\dot\gamma_0)$. The two theorems therefore
hold under $\Ric\geq0$ and $\Ric\geq-(n-1)Kg$, respectively.
In the latter case the elementary differential estimate reads
\[
 |\nabla\log u|^2-\alpha(\log u)_t
 \leq\frac{n\alpha^2}{2t}
 +\frac{n(n-1)\alpha^2K}{4(\alpha-1)},
\]
which is the Li-Yau-Davies estimate~\cite{Davies1989}.
\end{remark}

% END MAIN PROOFS
\bibliographystyle{alpha}
\bibliography{references}

\begin{thebibliography}{LNVV09}

\bibitem[AB79]{AronsonBenilan1979}
Donald~G. Aronson and Philippe B{\'e}nilan.
\newblock R{\'e}gularit{\'e} des solutions de l'{\'e}quation des milieux poreux dans {$\mathbb{R}^N$}.
\newblock {\em C. R. Acad. Sci. Paris S{\'e}r. A-B}, 288:A103--A105, 1979.

\bibitem[AB94]{AuchmutyBao1994}
Giles Auchmuty and David Bao.
\newblock Harnack-type inequalities for evolution equations.
\newblock {\em Proceedings of the American Mathematical Society}, 122(1):117--129, 1994.

\bibitem[AC09]{AndrewsClutterbuck2009}
Ben Andrews and Julie Clutterbuck.
\newblock Time-interior gradient estimates for quasilinear parabolic equations.
\newblock {\em Indiana University Mathematics Journal}, 58(1):351--380, 2009.

\bibitem[AC13]{AndrewsClutterbuck2013}
Ben Andrews and Julie Clutterbuck.
\newblock Sharp modulus of continuity for parabolic equations on manifolds and lower bounds for the first eigenvalue.
\newblock {\em Analysis \& PDE}, 6(5):1013--1024, 2013.

\bibitem[AFS08]{AzagraFerreraSanz2008}
Daniel Azagra, Juan Ferrera, and Beatriz Sanz.
\newblock Viscosity solutions to second order partial differential equations on {Riemannian} manifolds.
\newblock {\em Journal of Differential Equations}, 245:307--336, 2008.

\bibitem[AHS26]{AndrewsHauerSlegers2026}
Ben Andrews, Daniel Hauer, and Jessica Slegers.
\newblock A multi-point maximum principle to prove global {Harnack} inequalities for {Schr\"odinger} operators.
\newblock {\em Calculus of Variations and Partial Differential Equations}, 65, 2026.
\newblock Paper No. 112, 23 pp.

\bibitem[AT10]{ArmstrongTrokhimtchouk2010}
Scott~N. Armstrong and Maxim Trokhimtchouk.
\newblock Long-time asymptotics for fully nonlinear homogeneous parabolic equations.
\newblock {\em Calculus of Variations and Partial Differential Equations}, 38:521--540, 2010.

\bibitem[BBG17]{BakryBolleyGentil2017}
Dominique Bakry, Fran{\c c}ois Bolley, and Ivan Gentil.
\newblock The {Li-Yau} inequality and applications under a curvature-dimension condition.
\newblock {\em Annales de l'Institut Fourier}, 67(1):397--421, 2017.

\bibitem[BGM26]{BanerjeeGarofaloMahmoudian2026}
Agnid Banerjee, Nicola Garofalo, and Hamidreza Mahmoudian.
\newblock A nonlinear {Li-Yau} inequality and its consequences, 2026.
\newblock arXiv:2608.28881v2, 1 September 2026.

\bibitem[Cho92]{Chow1992}
Bennett Chow.
\newblock The {Yamabe} flow on locally conformally flat manifolds with positive {Ricci} curvature.
\newblock {\em Comm. Pure Appl. Math.}, 45(8):1003--1014, 1992.

\bibitem[CIL92]{CIL}
Michael~G. Crandall, Hitoshi Ishii, and Pierre-Louis Lions.
\newblock User's guide to viscosity solutions of second order partial differential equations.
\newblock {\em Bull. Amer. Math. Soc. (N.S.)}, 27(1):1--67, 1992.

\bibitem[Dav89]{Davies1989}
E.~B. Davies.
\newblock {\em Heat Kernels and Spectral Theory}, volume~92 of {\em Cambridge Tracts in Mathematics}.
\newblock Cambridge University Press, Cambridge, 1989.

\bibitem[EV90]{EstebanVazquez1990}
J.~R. Esteban and Juan~Luis V{\'a}zquez.
\newblock R{\'e}gularit{\'e} des solutions positives de l'{\'e}quation parabolique {$p$}-laplacienne.
\newblock {\em C. R. Acad. Sci. Paris S{\'e}r. I Math.}, 310:105--110, 1990.

\bibitem[Ham93a]{HamiltonRicci1993}
Richard~S. Hamilton.
\newblock The {Harnack} estimate for the {Ricci} flow.
\newblock {\em J. Differential Geom.}, 37(1):225--243, 1993.

\bibitem[Ham93b]{Hamilton1993}
Richard~S. Hamilton.
\newblock A matrix {Harnack} estimate for the heat equation.
\newblock {\em Communications in Analysis and Geometry}, 1(1):113--126, 1993.

\bibitem[HL11]{HarveyLawson2011}
F.~Reese Harvey and H.~Blaine Lawson, Jr.
\newblock Dirichlet duality and the nonlinear {Dirichlet} problem on {Riemannian} manifolds.
\newblock {\em Journal of Differential Geometry}, 88(3):395--482, 2011.

\bibitem[JK06]{JuutinenKawohl2006}
Petri Juutinen and Bernd Kawohl.
\newblock On the evolution governed by the infinity {Laplacian}.
\newblock {\em Mathematische Annalen}, 335(4):819--851, 2006.

\bibitem[JL25]{JeffresLi2025}
Thalia~D. Jeffres and Xiaolong Li.
\newblock Gradient bounds for viscosity solutions to certain elliptic equations.
\newblock {\em arXiv:2511.02073}, 2025.
\newblock Preprint.

\bibitem[KL14]{KimLee2014}
Soojung Kim and Ki-Ahm Lee.
\newblock Parabolic {Harnack} inequality of viscosity solutions on {Riemannian} manifolds.
\newblock {\em Journal of Functional Analysis}, 267(7):2152--2198, 2014.

\bibitem[KN09]{KotschwarNi2009}
Brett Kotschwar and Lei Ni.
\newblock Local gradient estimates of {$p$}-harmonic functions, {$1/H$}-flow, and an entropy formula.
\newblock {\em Ann. Sci. {\'E}c. Norm. Sup{\'e}r. (4)}, 42(1):1--36, 2009.

\bibitem[KS25]{KurkinenSiltakoski2025}
Tapio Kurkinen and Jarkko Siltakoski.
\newblock Intrinsic {Harnack}'s inequality for a general nonlinear parabolic equation in non-divergence form.
\newblock {\em Potential Analysis}, 62:393--437, 2025.

\bibitem[Li21]{Li2021}
Xiaolong Li.
\newblock Modulus of continuity estimates for fully nonlinear parabolic equations.
\newblock {\em Calculus of Variations and Partial Differential Equations}, 60, 2021.
\newblock Paper No. 182.

\bibitem[LNVV09]{LuNiVazquezVillani2009}
Peng Lu, Lei Ni, Juan~Luis V{\'a}zquez, and C{\'e}dric Villani.
\newblock Local {Aronson-B{\'e}nilan} estimates and entropy formulae for porous medium and fast diffusion equations on manifolds.
\newblock {\em J. Math. Pures Appl. (9)}, 91(1):1--19, 2009.

\bibitem[LTW24]{LiTuWang2024}
Xiaolong Li, Yucheng Tu, and Kui Wang.
\newblock On a class of quasilinear operators on smooth metric measure spaces.
\newblock {\em Communications in Analysis and Geometry}, 32(10):2757--2803, 2024.

\bibitem[LW17]{LiWang2017}
Xiaolong Li and Kui Wang.
\newblock Moduli of continuity for viscosity solutions on manifolds.
\newblock {\em Journal of Geometric Analysis}, 27:557--576, 2017.

\bibitem[LX11]{LiXu2011}
Junfang Li and Xiangjin Xu.
\newblock Differential {Harnack} inequalities on {Riemannian} manifolds {I}: Linear heat equation.
\newblock {\em Advances in Mathematics}, 226(5):4456--4491, 2011.

\bibitem[LY86]{LiYau1986}
Peter Li and Shing-Tung Yau.
\newblock On the parabolic kernel of the {Schr\"odinger} operator.
\newblock {\em Acta Mathematica}, 156:153--201, 1986.

\bibitem[Mos64]{Moser1964}
J{\"u}rgen Moser.
\newblock A {Harnack} inequality for parabolic differential equations.
\newblock {\em Communications on Pure and Applied Mathematics}, 17(1):101--134, 1964.

\bibitem[PV13]{PortilheiroVazquez2013}
Manuel Portilheiro and Juan~Luis V{\'a}zquez.
\newblock Degenerate homogeneous parabolic equations associated with the infinity-{Laplacian}.
\newblock {\em Calculus of Variations and Partial Differential Equations}, 46(3-4):705--724, 2013.

\bibitem[PV20]{ParviainenVazquez2020}
Mikko Parviainen and Juan~Luis V{\'a}zquez.
\newblock Equivalence between radial solutions of different parabolic gradient-diffusion equations and applications.
\newblock {\em Ann. Sc. Norm. Super. Pisa Cl. Sci. (5)}, 21:303--359, 2020.

\bibitem[SZ06]{SoupletZhang2006}
Philippe Souplet and Qi~S. Zhang.
\newblock Sharp gradient estimate and {Yau}'s {Liouville} theorem for the heat equation on noncompact manifolds.
\newblock {\em Bulletin of the London Mathematical Society}, 38(6):1045--1053, 2006.

\bibitem[Wan17]{Wang2017}
Yu-Zhao Wang.
\newblock Differential {Harnack} estimates and entropy formulae for weighted {$p$}-heat equations.
\newblock {\em Results in Mathematics}, 71:1499--1520, 2017.

\bibitem[Wan21]{Wang2021}
Yu-Zhao Wang.
\newblock {Perelman} type entropy formulae and differential {Harnack} estimates for weighted doubly nonlinear diffusion equations under curvature dimension condition.
\newblock {\em Bulletin of the Korean Mathematical Society}, 58(6):1539--1561, 2021.

\bibitem[WW20]{WangWang2020}
Yu-Zhao Wang and Xue-Ming Wang.
\newblock Difference {Harnack} estimates for weighted nonlinear reaction-diffusion equations on weighted {Riemannian} manifolds.
\newblock {\em Journal of Mathematics}, 40(6):653--661, 2020.

\bibitem[WYC13]{WangYangChen2013}
Yu-Zhao Wang, J.~Yang, and W.~Chen.
\newblock Gradient estimates and entropy formulae for weighted {$p$}-heat equations on smooth metric measure spaces.
\newblock {\em Acta Math. Sci. Ser. B Engl. Ed.}, 33(4):963--974, 2013.

\bibitem[Zha24]{ZhangSurvey}
Qi~S. Zhang.
\newblock Log gradient estimates for heat type equations on manifolds.
\newblock In Gang Tian, Qing Han, and Zhenlei Zhang, editors, {\em Surveys in Geometric Analysis 2023}, pages 186--212. Science Press, Beijing, 2024.

\bibitem[Zha25]{ZhangSharp2025}
Qi~S. Zhang.
\newblock A sharp {Li-Yau} gradient bound on compact manifolds.
\newblock {\em Communications in Analysis and Geometry}, 33(1):261--274, 2025.

\end{thebibliography}
\end{document}